\documentclass[12pt,reqno]{amsart}
\usepackage[utf8]{inputenc}
\usepackage[T1]{fontenc}

\usepackage{amsmath,amssymb,amsfonts}
\usepackage{mathtools}
\mathtoolsset{showonlyrefs}
\usepackage[expansion=false]{microtype}
\usepackage[margin=1in]{geometry}
\usepackage{xcolor}
\usepackage{tikz}     
\usepackage{lineno}   

\definecolor{lava}{rgb}{0.81, 0.06, 0.13}      
\definecolor{cobalt}{rgb}{0.0, 0.28, 0.67}     
\definecolor{tealgreen}{rgb}{0.0, 0.51, 0.5}   
\definecolor{lime}{HTML}{A6CE39}               

\usepackage[colorlinks,linkcolor=lava,citecolor=blue,urlcolor=cobalt,hypertexnames=false]{hyperref}
\hypersetup{unicode=true}

\numberwithin{equation}{section}

\numberwithin{bbb}{section}

\DeclareMathOperator{\Div}{div}

\newcommand{\N}{\mathbb N}

\newcommand{\R}{\mathbb R}

\newcommand{\Pp}{\mathbb P}   
\newcommand{\Qq}{\mathbb Q}   

\newtheorem{thm}{Theorem}[section]
\newtheorem{defi}{Definition}[section]
\newtheorem{lem}{Lemma}[section]
\newtheorem{prop}{Proposition}[section]

\newtheorem{rem}{Remark}[section]
\newtheorem{exam}{Example}[section]

\DeclareRobustCommand{\orcidicon}{%
    \begin{tikzpicture}
    \draw[lime, fill=lime] (0,0)
        circle [radius=0.16]
        node[white] {{\fontfamily{qag}\selectfont \tiny ID}};
    \draw[white, fill=white] (-0.0625,0.095)
        circle [radius=0.007];
    \end{tikzpicture}%
    \hspace{-2mm}%
}
\foreach \x in {A, B, C}{%
    \expandafter\xdef\csname orcid\x\endcsname{%
        \noexpand\href{https://orcid.org/\csname orcidauthor\x\endcsname}%
        {\noexpand\orcidicon}}%
}

\title[Penalized hyperviscous Navier--Stokes flows]{Penalized hyperviscous Navier--Stokes flows: well-posedness, incompressible limit and sharp decay}
\author[K.-M. Adeyemo, M. Majdoub and S. Pal]{Kabiru Michael Adeyemo, Mohamed Majdoub \& Subha Pal}
\address[M. Adeyemo]{Department of Mathematics, Hallmark University, Ijebu-Itele,
 Ogun State, Nigeria}
\email{mikyade2019@gmail.com}
\address[M. Majdoub]{Department of Mathematics, College of Science, Imam Abdulrahman Bin Faisal University, P. O. Box 1982, Dammam, Saudi Arabia.}
\address[M. Majdoub]{Basic and Applied Scientific Research Center, Imam Abdulrahman Bin Faisal University, P.O. Box 1982, 31441, Dammam, Saudi Arabia.}
\email{mmajdoub@iau.edu.sa}
\email{mohamed.majdoub@fst.rnu.tn}
\email{med.majdoub@gmail.com}
\address[S. Pal]{Department of Mathematics, IIT Palakkad, Kerala, India - 678623}
\email{subha@iitpkd.ac.in}
\email{sp234sp@gmail.com}

\begin{document}

\begin{abstract}
We study a three-dimensional Navier--Stokes system obtained by Temam's
penalty method in the presence of a biharmonic viscosity: the pressure
and the divergence constraint are replaced by the term
$-\varepsilon^{-1}\nabla\Div u$, and the convection term is written in
skew-symmetric form. We prove global well-posedness in $L^2(\R^3)$ and
global regularity for data of arbitrary size in $H^2(\R^3)$, with bounds
uniform in $\varepsilon$ and in time. For data in $L^1\cap L^2$, the
energy decays like $t^{-3/2}$, with a constant independent of
$\varepsilon$; for data in $L^1\cap H^2$, the first two derivatives decay
at the corresponding heat rates. We identify the leading asymptotic
profile and show that
$t^{3/2}\|u(t)\|_{L^2}^2\to\gamma_\varepsilon(8\pi\nu)^{-3/2}|m_\infty|^2$,
where $m_\infty$ is an asymptotic momentum and the explicit factor
$\gamma_\varepsilon\in(\frac23,1)$ measures the energy removed by the
penalization. Consequently, the decay rates are optimal in
$L^1\cap H^2$ under an explicit non-degeneracy condition. Finally, as
$\varepsilon\to0$, the solutions converge globally in time to those of
the incompressible hyperviscous Navier--Stokes equations at the rate
$\sqrt\varepsilon$, up to an explicit initial layer, and
$-\varepsilon^{-1}\Div u$ converges to the pressure.
\end{abstract}


\date{\today}

\maketitle

\section{Introduction and main results}
\label{Intro}

The motion of a viscous incompressible fluid filling the whole space is
described by the Navier--Stokes equations
\begin{equation}
  \label{NSE}
\begin{cases}
\partial_t u - \nu \Delta u + (u \cdot \nabla)u + \nabla p = 0, & \text{in } \R_+ \times \R^3, \\
\Div u = 0, & \text{in } \R_+\times \R^3, \\
u(0, x) = u_0(x), & \text{in } \R^3,
\end{cases}
\end{equation}
where $u(t,x)\in\R^3$ is the velocity, $p(t,x)\in\R$ the pressure and
$\nu>0$ the kinematic viscosity. Leray~\cite{Leray} and
Hopf~\cite{MN1951} constructed global weak solutions for arbitrary
finite-energy data; their uniqueness and regularity in three dimensions
remain open. Partial regularity was obtained by Caffarelli, Kohn and
Nirenberg~\cite{CKN1982} and Struwe~\cite{Struwe1988}, and we refer to
the monographs~\cite{Temam,3D-NS-Book} for the general theory and to
R\"ockner and Zhang~\cite{RZ2009} for a modified (tamed) system.

In numerical practice and in several models of continuum mechanics, the
constraint $\Div u=0$ is not imposed exactly. The system studied in this
paper combines three classical ways of relaxing or regularizing~\eqref{NSE}.
We describe them briefly, since together they explain the physical
meaning of the model.

\emph{Penalization and bulk viscosity.} In Temam's penalty
method~\cite{Temam1968} the pressure and the divergence constraint are
removed and replaced by the term $-\varepsilon^{-1}\nabla\Div u$, that
is, by the constitutive law $p=-\varepsilon^{-1}\Div u$. In continuum
mechanics, a normal stress proportional to the rate of change of volume
is a bulk (volume) viscous stress, so $\varepsilon^{-1}$ plays the role
of a large bulk viscosity. The linear part
$\nu\Delta u+\varepsilon^{-1}\nabla\Div u$ has the form of the Lam\'e
operator of a viscous fluid at constant density, in which the
coefficient of $\nabla\Div u$ is the bulk viscosity plus a multiple of
the shear viscosity; a large coefficient strongly damps changes of
volume. The limit $\varepsilon\to0$ is singular and formally leads back
to~\eqref{NSE}. The penalty method and the closely related artificial
compressibility method of Chorin~\cite{Chorin1967} are standard tools in
computational fluid dynamics, since they eliminate the pressure. We
refer to Shen~\cite{Shen1995} for error estimates, to Brefort, Ghidaglia
and Temam~\cite{BGT1988} for attractors of the penalized equations, to
Yamaguchi~\cite{Yamaguchi} for a justification on $\R^d$ at the
level of the Stokes and Navier--Stokes flows, and to Donatelli and
Marcati~\cite{DM2006} for the artificial compressibility approximation
on $\R^3$. The incompressible limit of genuinely compressible viscous
fluids was studied by Lions and Masmoudi~\cite{LM1998}.

\emph{Skew-symmetric convection.} Once $\Div u\neq0$, the convective
term $(u\cdot\nabla)u$ is no longer energy-neutral
(Example~\ref{ex:no-correction} below). Temam's correction replaces it
by
$$
(u\cdot\nabla)u+\tfrac12u\,\Div u
=\tfrac12\bigl[(u\cdot\nabla)u+\Div(u\otimes u)\bigr],
$$
that is, the average of the advective and conservative forms, which is
again energy-neutral. This skew-symmetric form is widely used in
energy-conserving discretizations~\cite{Zang1991}, and it also appears
in the work of Donatelli and Marcati~\cite{DM2006}. Plech\'a\v{c} and
\v{S}ver\'ak~\cite{Nonl2003} studied the one-parameter family
$$
\partial_tu+a\,(u\cdot\nabla)u+\tfrac12(1-a)\nabla|u|^2
+\tfrac12(\Div u)u=\Delta u+\varkappa\nabla\Div u,
\qquad a\in(0,1),\ \varkappa\ge0,
$$
as a model for~\eqref{NSE} without pressure. They proved that radially
symmetric, compactly supported data do not develop singularities in
dimensions $n\le4$, gave numerical evidence of blow-up for $n>4$, and
noted that the solutions should converge to those of~\eqref{NSE} as
$\varkappa\to\infty$. The nonlinearity considered here corresponds to the
endpoint $a=1$ of this family, with $\varkappa=\varepsilon^{-1}$. As
pointed out in~\cite{Nonl2003}, most questions concerning the full
regularity of such models in three dimensions appear to be open in the
absence of higher-order dissipation.

\emph{Biharmonic viscosity.} In the theory of multipolar (in particular
bipolar) viscous fluids of Ne\v{c}as and \v{S}ilhav\'y~\cite{NS1991}, the
stress depends on higher gradients of the velocity in a way that is
compatible with frame indifference and the Clausius--Duhem inequality;
in the linear case, this adds the term $\beta\Delta^2u$ to the momentum
equation. See~\cite{BBN1992,BB2014} for the incompressible theory
and~\cite{NNS1991} for compressible multipolar fluids. Hyperviscosity is
also used in turbulence simulations to concentrate dissipation at high
wave numbers~\cite{FKP2008}, in the same spirit as the spectral vanishing
viscosity of Tadmor~\cite{Tadmor1989}. For the incompressible equations,
Lions~\cite{Lions1969} proved global regularity when $-\Delta$ is
replaced by $(-\Delta)^{\gamma}$ with $\gamma\ge5/4$; refinements were
obtained by Katz and Pavlovi\'c~\cite{KP2002} and Tao~\cite{Tao2009}.
Fourth-order dissipation also appears in thin-film
growth~\cite{JMPA2015}, surface diffusion~\cite{IFB2014} and the
biharmonic heat equation on singular spaces~\cite{V2016}.

System~\eqref{Main-eq} below can therefore be viewed in two ways: as a
constant-density bipolar fluid with a large bulk viscosity, or as the
penalty approximation of the incompressible bipolar (hyperviscous)
Navier--Stokes equations. In both cases, the relevant questions are
whether the solutions are global, whether the estimates are uniform in
$\varepsilon$, how the penalization affects the long-time behavior, and
in what sense the incompressible equations are recovered as
$\varepsilon\to0$. This paper addresses these questions.

The present work also belongs to the study of Navier--Stokes equations
with additional dissipative or damping terms. A well-studied example is
the system with absorption term $\alpha|u|^{r-1}u$,
\begin{equation}
  \label{NSE-damping}
  \partial_t u - \nu\Delta u + (u\cdot\nabla)u + \alpha|u|^{r-1}u + \nabla p = 0,
  \qquad \Div u=0,
\end{equation}
which models flows through porous media~\cite{Ant2010}. Cai and
Jiu~\cite{JMAA2008} proved the global existence of weak solutions for
$r\ge1$ and of strong solutions for $r\ge7/2$. Uniqueness was obtained by
Zhang, Wu and Lu~\cite{Zhang} for $3<r\le5$ and by Zhou~\cite{Zhou2012}
for $r=3$. We also refer to Hajduk and Robinson~\cite{JDE2017} for the
critical Brinkman--Forchheimer equations on the torus, to Kalantarov and
Zelik~\cite{CPAA2012} for attractors, to~\cite{JEE24,Xu} for strong
solutions, to~\cite{Su24,ZAMP2020,Zhou2026} for decay, and to Pal and
Haloi~\cite{PH2021} for bounded domains with Navier boundary conditions.
Flows of fluids with pressure-dependent viscosity in unbounded domains
were studied by Bul\'{\i}\v{c}ek, Majdoub and M\'alek~\cite{BMM2010}.

The study of the decay of solutions to~\eqref{NSE} goes back to
Schonbek~\cite{Schonbek1985,Schonbek1986}, who introduced the Fourier
splitting method, and to Kajikiya and Miyakawa~\cite{KM1986} and
Wiegner~\cite{Wiegner1987}. Higher derivatives were treated by Schonbek
and Wiegner~\cite{SW1996}, lower bounds by Miyakawa and
Schonbek~\cite{MS2001}, and asymptotic profiles by Carpio~\cite{Carpio1996}
and by Fujigaki and Miyakawa~\cite{FM2001}. For compressible fluids, the
decay theory goes back to Matsumura and Nishida~\cite{MN1980}, and the
asymptotic profile splits into a heat wave and diffusion waves; see Hoff
and Zumbrun~\cite{HZ1995}. A similar splitting between the solenoidal and
gradient components appears in Theorem~\ref{thm:profile} below.

\subsection{The model}

In this paper, we consider the Cauchy problem
\begin{equation}
    \label{Main-eq}
    \partial_t u
    - \nu \Delta u
    + \mathcal{N}(u)
    - \frac{1}{\varepsilon}\nabla \Div u
    + \beta \Delta^2 u
    = 0,
    \qquad
    (t,x) \in \R_+ \times \R^3,
    \qquad u(0)=u_0,
\end{equation}
where $\nu>0$, $\beta>0$, $\varepsilon>0$, and
\begin{equation}
    \label{def:N}
    \mathcal{N}(u)
    := (u\cdot\nabla)u + \frac12\,u\,\Div u .
\end{equation}
Formally, taking the $L^2$ inner product of~\eqref{Main-eq} with $u$
and using the cancellation property of $\mathcal N$
(Lemma~\ref{lem:cancel}), we obtain the energy identity
\begin{equation}
    \label{eq:formal-energy}
    \frac12\frac{d}{dt}\|u(t)\|_{L^2}^2
    + \nu\|\nabla u(t)\|_{L^2}^2
    + \beta\|\Delta u(t)\|_{L^2}^2
    + \frac1\varepsilon\|\Div u(t)\|_{L^2}^2
    = 0.
\end{equation}
In particular, $\int_0^\infty\|\Div u\|_{L^2}^2\,dt\le
\frac12\min\{\varepsilon,\nu^{-1}\}\|u_0\|_{L^2}^2$, so the penalization
drives the flow toward incompressibility. Note also that this energy
estimate is unconditional, in contrast with the compressible
Navier--Stokes system. The main difficulty in the analysis
of~\eqref{Main-eq} is uniformity with respect to $\varepsilon$: the term
$\varepsilon^{-1}\nabla\Div u$ is bounded only by $\varepsilon^{-1/2}$ in
$L^2(0,T;H^{-1})$, so none of our bounds may depend on it unfavorably.
The structural fact behind all our estimates is that the penalization
enters every energy identity with a favorable sign and acts only on the
gradient part of $u$ (Section~\ref{subsec:semigroup}).

\subsection{Main results}

In this subsection, we state our main results. We first make precise
the notion of weak solution used throughout the paper.

\begin{defi}
\label{def:weak-sol}
Let $u_0\in L^2(\R^3;\R^3)$ and $\nu,\beta,\varepsilon>0$. A vector
field $u$ is a \emph{global weak solution} of~\eqref{Main-eq} if the
following hold.
\begin{itemize}
    \item[\textnormal{(i)}] For every $T>0$,
    $u \in L^\infty(0,T;L^2(\R^3))\cap L^2(0,T;H^2(\R^3))$.

    \item[\textnormal{(ii)}] For every $T>0$ and every
    $\phi\in C_c^\infty([0,T)\times\R^3;\R^3)$,
    \begin{align}
        \label{eq:weak-form}
        &-\int_0^T\!\!\int_{\R^3} u\cdot\partial_t\phi\,dx\,dt
        + \nu\int_0^T\!\!\int_{\R^3}\nabla u:\nabla\phi\,dx\,dt
        + \beta\int_0^T\!\!\int_{\R^3}\Delta u\cdot\Delta\phi\,dx\,dt
        \notag\\
        &\qquad
        + \frac1\varepsilon\int_0^T\!\!\int_{\R^3}
            (\Div u)(\Div\phi)\,dx\,dt
        + \int_0^T\!\!\int_{\R^3}
            \mathcal{N}(u)\cdot\phi\,dx\,dt
        = \int_{\R^3}u_0(x)\cdot\phi(0,x)\,dx.
    \end{align}

    \item[\textnormal{(iii)}] For almost every $t\ge 0$,
    \begin{equation}
        \label{eq:energy-ineq}
        \|u(t)\|_{L^2}^2
        + 2\int_0^t\Bigl(\nu\|\nabla u\|_{L^2}^2
        + \beta\|\Delta u\|_{L^2}^2
        + \frac{1}{\varepsilon}\|\Div u\|_{L^2}^2\Bigr)d\tau
        \le \|u_0\|_{L^2}^2.
    \end{equation}
\end{itemize}
\end{defi}

\begin{rem}\rm
\label{rem:weak-solution-comments}
\leavevmode
\begin{itemize}
    \item[(a)] The test functions in~\eqref{eq:weak-form} are not
    required to be divergence-free. This is the natural weak formulation
    of~\eqref{Main-eq}, and it is one of the simplifications gained by
    penalizing the constraint rather than imposing it.

    \item[(b)] Item (i) implies $u\in L^4(0,T;L^4(\R^3))$, since
    $\|u\|_4^4\le\|u\|_2^2\|u\|_\infty^2\le C\|u\|_2^2\|u\|_{H^2}^2$.

    \item[(c)] Item (i) also implies
    $\mathcal N(u)\in L^1(0,T;L^1)\cap L^2(0,T;H^{-2})$, so that all
    integrals in~\eqref{eq:weak-form} are finite and
    $\partial_tu\in L^2(0,T;H^{-2})$ (Lemma~\ref{lem:N-bounds} and
    Proposition~\ref{prop:weak-unique}).
\end{itemize}
\end{rem}

Our first result establishes global well-posedness in the energy space.

\begin{thm}
\label{thm:wp}
Let $u_0\in L^2(\R^3;\R^3)$ and $\nu,\beta,\varepsilon>0$. There exists
a unique global weak solution $u$ of~\eqref{Main-eq}. It belongs to
$C([0,\infty);L^2(\R^3))$, satisfies $\partial_tu\in
L^2(0,T;H^{-2}(\R^3))$ and $u\in L^4(0,T;L^4(\R^3))$ for every $T>0$,
and the energy equality holds for every $t\ge0$:
\begin{equation}
\label{eq:energy-equality}
\|u(t)\|_{2}^2
+ 2\int_0^t\Bigl(\nu\|\nabla u\|_{2}^2
+ \beta\|\Delta u\|_{2}^2
+ \frac{1}{\varepsilon}\|\Div u\|_{2}^2\Bigr)d\tau
= \|u_0\|_{2}^2.
\end{equation}
In addition, if $v$ denotes the weak solution with initial data
$v_0\in L^2$, then for every $T>0$ there exists a constant $C_T$,
depending on $T$, $\nu$, $\beta$, $\|u_0\|_2$ and $\|v_0\|_2$ but not on
$\varepsilon$, such that
$\sup_{0\le t\le T}\|u(t)-v(t)\|_2\le C_T\|u_0-v_0\|_2$.
\end{thm}

Our second result provides global strong solutions for $H^2$ data of
arbitrary size, with bounds uniform in $\varepsilon$ and in time. It
involves the quantity $\|u\|_2\|\nabla u\|_2$, which is critical
for~\eqref{NSE} and controls $\|u\|_{\dot H^{1/2}}^2$.

\begin{thm}
\label{thm:uniform-small-data-strong}
Let $\nu,\beta,\varepsilon>0$ and $u_0\in H^2(\R^3;\R^3)$. The weak
solution of Theorem~\ref{thm:wp} is a strong solution,
$$
u\in C([0,\infty);H^2(\R^3))\cap L^2_{\mathrm{loc}}([0,\infty);H^4(\R^3)),
$$
and it satisfies
\begin{equation}
\label{eq:large-data-bound}
\sup_{t\ge0}\|u(t)\|_{H^2}^2
+\int_0^\infty\Bigl(\nu\|\nabla u\|_{H^2}^2+\beta\|\Delta u\|_{H^2}^2
+\frac1\varepsilon\|\Div u\|_{H^2}^2\Bigr)dt
\le 3\Lambda_0\|u_0\|_{H^2}^2,
\end{equation}
where $\Lambda_0=\exp\bigl(c\,\|u_0\|_2^4/(\beta\nu^3)\bigr)$ and $c$ is
an absolute constant. In addition, there exists an absolute constant
$c_0>0$ such that, with $\delta_0:=c_0\nu^2$, the condition
$\|u_0\|_2\|\nabla u_0\|_2\le\delta_0$ implies that
$t\mapsto\|\nabla^ku(t)\|_2$ is nonincreasing for $k=0,1,2$ and that
\begin{equation}
\label{eq:uniform-H2-bound}
\sup_{t\ge0}\|u(t)\|_{H^2}^2
+
\int_0^\infty
\Bigl(
\nu\|\nabla u(t)\|_{H^2}^2
+
2\beta\|\Delta u(t)\|_{H^2}^2
+
\frac2\varepsilon\|\Div u(t)\|_{H^2}^2
\Bigr)\,dt
\le
\|u_0\|_{H^2}^2 .
\end{equation}
\end{thm}

\begin{rem}\rm
\label{rem:uniform-epsilon-strong}
None of the constants in Theorem~\ref{thm:uniform-small-data-strong}
depends on $\varepsilon$. In each energy identity, the penalization
contributes the nonnegative term $\varepsilon^{-1}\|\nabla^k\Div u\|_2^2$,
which is either kept or discarded and is never estimated from above.
Note also that the smallness threshold $\delta_0$ depends only on $\nu$,
and that the condition $\|u_0\|_2\|\nabla u_0\|_2\le\delta_0$ is
invariant under the scaling $u_0\mapsto\lambda u_0(\lambda\,\cdot)$
associated with~\eqref{NSE}; it is much weaker than smallness in $H^2$.
The key ingredient is Lemma~\ref{lem:Hk-key}, in which the nonlinear
terms are controlled by the critical norm $\|u\|_{\dot H^{1/2}}$ times
the dissipation.
\end{rem}

Our third result concerns the large-time decay of solutions, without
any smallness assumption.

\begin{thm}
\label{thm:optimal-decay}
Let $\nu,\beta,\varepsilon>0$.
\begin{itemize}
\item[\textnormal{(i)}] If $u_0\in L^1\cap L^2(\R^3;\R^3)$, the weak
solution satisfies
\begin{equation}
\label{eq:L2-decay-final}
\|u(t)\|_{2}\le C_0(1+t)^{-3/4},\qquad t\ge0,
\end{equation}
where $C_0$ depends only on $\nu$, $\|u_0\|_1$ and $\|u_0\|_2$, is
nondecreasing in the last two arguments, and is independent of $\beta$
and $\varepsilon$. If $\|u_0\|_1+\|u_0\|_2\le1$, then
$C_0\le C_\nu(\|u_0\|_1+\|u_0\|_2)$.
\item[\textnormal{(ii)}] If $u_0\in L^1\cap H^2(\R^3;\R^3)$, then, for
$k=1,2$,
\begin{equation}
\label{eq:derivative-decay-final}
\|\nabla^ku(t)\|_{2}
\le
C(1+t)^{-3/4-k/2},\qquad t\ge0,
\end{equation}
where $C$ depends only on $\nu$, $\beta$, $\|u_0\|_1$ and
$\|u_0\|_{H^2}$, and not on $\varepsilon$.
\end{itemize}
\end{thm}

To describe the asymptotic profile, we denote by $\Pp$ the Leray
projector and by $\Qq=I-\Pp$ the projector onto gradient fields. For
$m\in\R^3$ and $t>0$, we let
$K_\varepsilon(t)m:=S_\varepsilon(t)(m\,\boldsymbol\delta)$ denote the
linear evolution of a point momentum, where $\boldsymbol\delta$ is the
Dirac mass at the origin; its Fourier transform is given
in~\eqref{eq:Keps} below. We also set
\begin{equation}
\label{eq:gamma-eps}
\gamma_\varepsilon:=\frac23+\frac13\Bigl(\frac{\varepsilon\nu}{1+\varepsilon\nu}\Bigr)^{3/2}\in\Bigl(\frac23,1\Bigr),
\qquad
\kappa_\varepsilon:=\Bigl(\tfrac12\min\{\varepsilon,\nu^{-1}\}\Bigr)^{1/2}.
\end{equation}

\begin{thm}
\label{thm:profile}
Let $\nu,\beta,\varepsilon>0$ and $u_0\in L^1\cap L^2(\R^3;\R^3)$. Let
$u$ be the corresponding weak solution, and let $C_0$ be the constant
in~\eqref{eq:L2-decay-final}.
\begin{itemize}
\item[\textnormal{(i)}] The vector
\begin{equation}
\label{eq:m-infty}
m_\infty:=\int_{\R^3}u_0\,dx+\frac12\int_0^\infty\!\!\int_{\R^3}u\,\Div u\,dx\,dt
\end{equation}
is well defined, $|m_\infty-\int u_0|\le 2^{-1/2}C_0\kappa_\varepsilon\|u_0\|_2$, and
\begin{equation}
\label{eq:profile}
\lim_{t\to\infty}t^{3/4}\bigl\|u(t)-K_\varepsilon(t)m_\infty\bigr\|_2=0,
\qquad
\lim_{t\to\infty}t^{3/2}\|u(t)\|_2^2=\gamma_\varepsilon\,(8\pi\nu)^{-3/2}|m_\infty|^2 .
\end{equation}
In particular, $\|u(t)\|_2$ decays exactly like $t^{-3/4}$ if
$m_\infty\ne0$, and $\|u(t)\|_2=o(t^{-3/4})$ if $m_\infty=0$.
\item[\textnormal{(ii)}] There are absolute constants $\Lambda\ge1$ and
$c_*>0$ with the following property. If
\begin{equation}
\label{eq:nondeg}
\Bigl|\int_{\R^3}u_0\,dx\Bigr|\ge\Lambda\,\kappa_\varepsilon\,C_0\,\|u_0\|_2,
\end{equation}
then $\|u(t)\|_2\ge c_*\bigl|\int u_0\bigr|(\nu t)^{-3/4}$ for all
$t\ge T_1$, where $T_1\ge1$ depends only on $\nu$, $\beta$, $C_0$,
$\int u_0$ and the modulus of continuity of $\widehat{u_0}$ at the
origin, and not on $\varepsilon$.
\item[\textnormal{(iii)}] If in addition $u_0\in H^2$, then
under~\eqref{eq:nondeg} there exist $c_1,c_2>0$ and $T_2\ge T_1$,
independent of $\varepsilon$, such that
$\|\nabla^ku(t)\|_2\ge c_k(1+t)^{-3/4-k/2}$ for $k=1,2$ and $t\ge T_2$.
\end{itemize}
\end{thm}

\begin{rem}
\label{rem:sharpness}\rm
Theorem~\ref{thm:profile} shows that the rates of
Theorem~\ref{thm:optimal-decay} are attained \emph{within the class
$L^1\cap H^2$ of that theorem}, with lower-bound constants that are
uniform in $\varepsilon$ in the range where~\eqref{eq:nondeg} holds.
Since $\kappa_\varepsilon\le\min\{(\varepsilon/2)^{1/2},(2\nu)^{-1/2}\}$,
condition~\eqref{eq:nondeg} is satisfied in two typical situations. First, for
fixed $u_0$ with $\int u_0\ne0$, it holds for all
$\varepsilon\le\varepsilon_*:=2\bigl(|\int u_0|/(\Lambda
C_0\|u_0\|_2)\bigr)^2$. Second, for $u_0=\lambda\varphi$ with
$\varphi\in L^1\cap H^2$, $\int\varphi\ne0$ and $\lambda>0$ small, it
holds for all $\varepsilon>0$: the left side of~\eqref{eq:nondeg} is of
order $\lambda$ and, by Theorem~\ref{thm:optimal-decay}(i), the right
side is of order $\lambda^2$.

The data in these examples are not divergence-free, and this is essential.
If $u_0\in L^1$ and $\Div u_0=0$, then $\xi\cdot\widehat{u_0}(\xi)=0$
and the continuity of $\widehat{u_0}$ give $\int u_0=0$, so that the
linear part $S_\varepsilon(t)u_0$ decays like $o(t^{-3/4})$. For such
data, $m_\infty=\frac12\int_0^\infty\!\int u\,\Div u$ is of order
$\sqrt\varepsilon$, and whether it vanishes depends on the nonlinear
evolution. An earlier version of this paper used divergence-free data with
$|\widehat{u_0}|$ bounded below by a positive constant near the origin.
Such data do not belong to $L^1$, so that example did not establish
optimality within the class of Theorem~\ref{thm:optimal-decay}.
\end{rem}

\begin{rem}\rm
\label{rem:gamma}
The factor $\gamma_\varepsilon$ admits a simple interpretation. The point momentum
$m_\infty\boldsymbol\delta$ splits into its solenoidal part $\Pp(m_\infty\boldsymbol\delta)$,
which diffuses with viscosity $\nu$, and its gradient part
$\Qq(m_\infty\boldsymbol\delta)$, which diffuses with the larger viscosity
$\nu+\varepsilon^{-1}$. By the proof of Lemma~\ref{lem:kernel}, their
contributions to the limit in~\eqref{eq:profile} are
$$
\frac23(8\pi\nu)^{-3/2}|m_\infty|^2
\qquad\text{and}\qquad
\frac13\bigl(8\pi(\nu+\varepsilon^{-1})\bigr)^{-3/2}|m_\infty|^2 .
$$
As $\varepsilon\to0$, the gradient part is damped out and
$\gamma_\varepsilon\downarrow\frac23$; as $\varepsilon\to\infty$, the
penalization disappears and $\gamma_\varepsilon\uparrow1$, which is the
value for the heat equation. The
biharmonic term does not affect the leading asymptotics: its symbol
$\beta|\xi|^4$ is negligible against $\nu|\xi|^2$ at the scale
$|\xi|\sim t^{-1/2}$. The vector $m_\infty$ is the total momentum at
infinite time. Indeed, since $\int\mathcal N(u)\,dx=-\frac12\int u\,\Div u\,dx$,
the momentum $\int u(t)\,dx$ is not conserved, and formally
$m_\infty=\lim_{t\to\infty}\int u(t)\,dx$.
\end{rem}

\begin{rem}\rm
\label{rem:comparison-damping}
In the decay theory of~\eqref{NSE} and of the damped
system~\eqref{NSE-damping}~\cite{Schonbek1985,Schonbek1986,Wiegner1987,Su24,ZAMP2020,Zhou2026},
the leading mechanism is the heat flow acting on low frequencies, while
the additional structure enters only through the control of the
nonlinear and high-frequency terms. Two new features appear here. First,
since $u$ is not divergence-free, the Duhamel formula contains the term
$u\,\Div u$, whose contribution is controlled by the penalization
through $\kappa_\varepsilon$. Second, the penalization changes the
leading profile itself, through $\gamma_\varepsilon$.
\end{rem}

Our last result concerns the incompressible limit $\varepsilon\to0$.
Let $v$ be the solution of the incompressible hyperviscous
Navier--Stokes equations
\begin{equation}
\label{eq:HNS}
\partial_tv-\nu\Delta v+\beta\Delta^2v+(v\cdot\nabla)v+\nabla\pi=0,
\qquad \Div v=0,\qquad v(0)=\Pp u_0,
\end{equation}
which is global and smooth for every $\Pp u_0\in H^2$; see~\cite{Lions1969}
and Proposition~\ref{prop:HNS}.

\begin{thm}
\label{thm:incompressible-limit}
Let $\nu,\beta>0$ and $u_0\in H^2(\R^3;\R^3)$. For $\varepsilon>0$ let
$u^\varepsilon$ be the strong solution of~\eqref{Main-eq}, let $v$ be
the solution of~\eqref{eq:HNS}, and let
$\ell^\varepsilon(t):=\Qq S_\varepsilon(t)u_0$. There exists a constant
$C$, depending only on $\nu$, $\beta$ and $\|u_0\|_{H^2}$, such that the
following hold for every $\varepsilon>0$.
\begin{itemize}
\item[\textnormal{(i)}]
$\displaystyle
\sup_{t\ge0}\|u^\varepsilon(t)-v(t)-\ell^\varepsilon(t)\|_2
+\Bigl(\int_0^\infty\|\nabla(u^\varepsilon-v)\|_2^2\,dt\Bigr)^{1/2}
\le C\sqrt\varepsilon\,\|u_0\|_2 ;
$
\item[\textnormal{(ii)}] the initial layer satisfies
$\|\ell^\varepsilon(t)\|_2\le\|e^{(t/\varepsilon)\Delta}\Qq u_0\|_2\le\|\Qq u_0\|_2$,
and, if $u_0\in L^1$, also
$\|\ell^\varepsilon(t)\|_2\le(8\pi t/\varepsilon)^{-3/4}\|u_0\|_1$;
\item[\textnormal{(iii)}] as $\varepsilon\to0$,
$\nabla\bigl(-\varepsilon^{-1}\Div u^\varepsilon\bigr)\to\nabla\pi$ in
$\mathcal D'((0,\infty)\times\R^3)$.
\end{itemize}
Consequently, $u^\varepsilon\to v$ in $L^\infty(t_0,\infty;L^2)$ for every
$t_0>0$, in $L^p(0,T;L^2)$ for every $p<\infty$ and $T>0$, and in
$L^2(0,\infty;\dot H^1)$ with rate $\sqrt\varepsilon$. If
$\Div u_0=0$, then $\ell^\varepsilon\equiv0$ and
$\sup_{t\ge0}\|u^\varepsilon(t)-v(t)\|_2\le C\sqrt\varepsilon\|u_0\|_2$.
\end{thm}

\begin{rem}\rm
\label{rem:limit}
In the presence of the biharmonic viscosity,
Theorem~\ref{thm:incompressible-limit} makes rigorous the heuristic
limit $\varkappa\to\infty$ mentioned in~\cite{Nonl2003}. The convergence
is global in time. The only obstruction to uniform convergence on
$[0,\infty)$ is the initial layer $\ell^\varepsilon$, which is the linear
evolution of the gradient part of the data with diffusivity of order
$\varepsilon^{-1}$. This layer acts on the time scale $t\sim\varepsilon$
and vanishes for well-prepared data. The rate $\sqrt\varepsilon$ in~(i) is the one suggested by the
energy bound $\int_0^\infty\|\Div u^\varepsilon\|_2^2\,dt\le\frac\varepsilon2\|u_0\|_2^2$.
For error estimates of the penalty method we refer
to~\cite{Shen1995}.
\end{rem}

\subsection{Strategy of the proofs}

The proof of Theorem~\ref{thm:wp} relies on a Friedrichs approximation
and the Aubin--Lions--Simon lemma. The approximation must not involve the Leray
projector (Remark~\ref{rem:no-leray}). Uniqueness and the energy
equality follow from the observation that $L^\infty L^2\cap L^2H^2$
already gives $\partial_tu\in L^2H^{-2}$, which is the dual of the
energy space.

For Theorem~\ref{thm:uniform-small-data-strong}, the main point is
Lemma~\ref{lem:Hk-key}: the nonlinear terms in the $\dot H^1$ and
$\dot H^2$ energy identities are bounded by
$C\|u\|_{\dot H^{1/2}}\|\nabla^{k+1}u\|_2^2$. Combined with the biharmonic
dissipation, this gives Gronwall bounds for arbitrary data. Combined with
the viscous dissipation, it shows that the $\dot H^1$ and $\dot H^2$ norms
are nonincreasing as soon as $\|u\|_2\|\nabla u\|_2$ is small, a
condition that is reached after a time controlled by the energy.

Theorem~\ref{thm:optimal-decay} is proved by Schonbek's Fourier
splitting method. The low-frequency part of $\widehat u$ is estimated
through the Duhamel formula and the decomposition~\eqref{eq:N-decomp};
for large data, we use a finite iteration on the time integral of the
energy. The
derivative decay follows from the monotone differential inequalities of
Theorem~\ref{thm:uniform-small-data-strong}.

The proof of Theorem~\ref{thm:profile} uses the splitting of the linear
semigroup into its solenoidal and gradient parts and a standard
dominated convergence argument in the Duhamel formula. The lower bound
relies on the fact that the solenoidal part of the semigroup does not
depend on $\varepsilon$. Theorem~\ref{thm:incompressible-limit} combines an energy
estimate for $\Pp u^\varepsilon-v$, which is driven only by
$\Div u^\varepsilon$, with an estimate of $\Qq u^\varepsilon$ in the
Duhamel formula.

\subsection{Organization and notation}

The structure of the paper is as follows. In
Section~\ref{sec:prelim}, we collect the functional inequalities, the
properties of the nonlinearity and the description of the linear
semigroup. Section~\ref{sec:existence} is devoted to the proof of
Theorem~\ref{thm:wp}, and Section~\ref{sec:strong-solutions} to the
proof of Theorem~\ref{thm:uniform-small-data-strong}. The decay
estimates of Theorem~\ref{thm:optimal-decay} are proved in
Section~\ref{sec:optimal-decay}, and the asymptotic profile of
Theorem~\ref{thm:profile} in Section~\ref{sec:profile}. In
Section~\ref{sec:limit}, we prove Theorem~\ref{thm:incompressible-limit}.
Finally, Section~\ref{sec:conclusion} contains concluding remarks and
open questions.

We write $\|\cdot\|_p:=\|\cdot\|_{L^p(\R^3)}$. The Fourier transform is
normalized by
$$
\widehat f(\xi)=(2\pi)^{-3/2}\int_{\R^3} e^{-ix\cdot\xi}f(x)\,dx,
$$
so that $\|f\|_2=\|\widehat f\|_2$, $|\widehat f(\xi)|\le(2\pi)^{-3/2}\|f\|_1$ and
$\widehat f(0)=(2\pi)^{-3/2}\int f$. For an integer $k\ge0$, $\nabla^kf$
is the tensor of all derivatives of order $k$, so that
$\|\nabla^kf\|_2^2=\int|\xi|^{2k}|\widehat f|^2\,d\xi$; in particular
$\|\nabla^2f\|_2=\|\Delta f\|_2$ and $\|\nabla^3f\|_2=\|\nabla\Delta f\|_2$.
For integer $s\ge0$ we use the norm
$\|f\|_{H^s}^2:=\sum_{k=0}^s\|\nabla^kf\|_2^2$, and we write
$\|f\|_{\dot H^s}:=\||\xi|^s\widehat f\|_2$ for real $s\ge0$. We write
$A:B$ for the Frobenius product of matrices. Throughout the paper, $C$
denotes a positive constant, independent of $\varepsilon$, whose
dependence on the other parameters is indicated and whose value may
change from line to line.

\section{Background and auxiliary results}
\label{sec:prelim}

This section collects the functional inequalities, the properties of the
nonlinearity and the description of the linear semigroup that are used
throughout the paper.

\subsection{Functional inequalities}

\begin{lem}
\label{lem:ineq}
There is an absolute constant $C$ such that, for functions on $\R^3$,
\begin{align}
&\|f\|_3\le C\|f\|_{\dot H^{1/2}}\le C\|f\|_2^{1/2}\|\nabla f\|_2^{1/2},
\qquad
\|f\|_6\le C\|\nabla f\|_2,
\label{eq:ineq-sobolev}\\
&\|f\|_\infty\le C\|f\|_2^{1/4}\|\Delta f\|_2^{3/4},
\qquad
\|f\|_\infty\le C\|\nabla f\|_2^{1/2}\|\nabla^2f\|_2^{1/2},
\label{eq:ineq-agmon}\\
&\|\nabla f\|_4^2\le C\|f\|_{\dot H^{1/2}}\|\nabla^3 f\|_2,
\qquad
\|\nabla^kf\|_2^2\le\|\nabla^{k-1}f\|_2\|\nabla^{k+1}f\|_2\ (k\ge1).
\label{eq:ineq-interp}
\end{align}
\end{lem}

\begin{proof}
The Sobolev embeddings $\dot H^s(\R^3)\hookrightarrow L^{6/(3-2s)}(\R^3)$,
$0\le s<3/2$, give $\|f\|_3\le C\|f\|_{\dot H^{1/2}}$,
$\|f\|_6\le C\|\nabla f\|_2$ and
$\|\nabla f\|_4\le C\|f\|_{\dot H^{7/4}}$. H\"older's inequality on the
Fourier side gives
$\|f\|_{\dot H^{s}}\le\|f\|_{\dot H^{s_0}}^{1-\theta}\|f\|_{\dot H^{s_1}}^{\theta}$
for $s=(1-\theta)s_0+\theta s_1$; this yields the second inequality
in~\eqref{eq:ineq-sobolev}, the last inequality in~\eqref{eq:ineq-interp},
and $\|f\|_{\dot H^{7/4}}^2\le\|f\|_{\dot H^{1/2}}\|f\|_{\dot H^3}$.
For~\eqref{eq:ineq-agmon} we split
$(2\pi)^{3/2}\|f\|_\infty\le\|\widehat f\|_1$ at $|\xi|=R$ and use the
Cauchy--Schwarz inequality:
$\|\widehat f\|_1\le CR^{3/2}\|f\|_2+CR^{-1/2}\|\Delta f\|_2$ and
$\|\widehat f\|_1\le CR^{1/2}\|\nabla f\|_2+CR^{-1/2}\|\nabla^2f\|_2$;
optimizing in $R$ gives the claim.
\end{proof}

\subsection{The nonlinearity}

We begin with the cancellation property of the skew-symmetric
form~\eqref{def:N}.

\begin{lem}
\label{lem:cancel}
For every $u\in H^2(\R^3;\R^3)$,
\begin{equation}
    \label{eq:cancel}
    \int_{\R^3} \mathcal{N}(u)\cdot u\,dx = 0.
\end{equation}
\end{lem}

\begin{proof}
For $u\in C_c^\infty$, integration by parts gives
$$
\int(u\cdot\nabla)u\cdot u\,dx=\frac12\int(u\cdot\nabla)|u|^2\,dx
=-\frac12\int|u|^2\Div u\,dx
=-\int\frac12u\,\Div u\cdot u\,dx .
$$
Both sides are continuous on $H^2$, since
$\int|u|^2|\nabla u|\,dx\le\|u\|_4^2\|\nabla u\|_2$, and $C_c^\infty$ is
dense in $H^2$.
\end{proof}

The following example shows that the correction term is needed: for
compressible fields, the convective term alone does not
satisfy~\eqref{eq:cancel}.

\begin{exam}[Failure of the cancellation without the Temam correction]
\label{ex:no-correction}
Let $\chi(x):=e^{-|x|^2}$ and $u:=\nabla\chi$, so that
$\Div u=\Delta\chi=(4|x|^2-6)e^{-|x|^2}$. Using
$\int_{\R^3}(u\cdot\nabla)u\cdot u\,dx
=-\frac12\int_{\R^3}|u|^2\Div u\,dx$ and Gaussian moments,
$$
\int_{\R^3}|\nabla\chi|^2\,\Delta\chi\,dx
=16\pi\int_0^\infty r^4\bigl(4r^2-6\bigr)e^{-3r^2}\,dr
=-\frac{16\pi}{9}\sqrt{\frac{\pi}{3}},
$$
so that
$$
\int_{\R^3}(u\cdot\nabla)u\cdot u\,dx
=\frac{8\pi}{9}\sqrt{\frac{\pi}{3}}>0 .
$$
The convective term alone injects energy for this field.
\end{exam}

\begin{rem}\rm
\label{rem:N-decomp}
Since $(u\cdot\nabla)u=\Div(u\otimes u)-u\,\Div u$, we have
\begin{equation}
    \label{eq:N-decomp}
    \mathcal{N}(u)
    = \Div(u\otimes u)
      - \frac12\,u\,\Div u.
\end{equation}
The divergence-form part contributes a factor $|\xi|$ on the Fourier
side, while the compressible part is controlled by the penalization
dissipation. This identity is used in the compactness argument of
Section~\ref{sec:existence} and in the low-frequency analysis of
Sections~\ref{sec:optimal-decay} and~\ref{sec:profile}.
\end{rem}

\begin{lem}
\label{lem:N-bounds}
Let $c_N:=1+\sqrt3/2$. For $u\in H^1(\R^3;\R^3)$, we have
\begin{equation}
\label{eq:N-pointwise}
\begin{gathered}
|\mathcal N(u)|\le c_N|u||\nabla u|,
\qquad
\|\mathcal N(u)\|_1\le c_N\|u\|_2\|\nabla u\|_2,\\
\int_{\R^3}\mathcal N(u)\,dx=-\frac12\int_{\R^3}u\,\Div u\,dx.
\end{gathered}
\end{equation}
Furthermore, for every $\xi\in\R^3$, we have
\begin{equation}
\label{eq:Nhat-low-frequency}
|\widehat{\mathcal N(u)}(\xi)|\le(2\pi)^{-3/2}\Bigl(|\xi|\,\|u\|_2^2+\frac12\|u\|_2\|\Div u\|_2\Bigr).
\end{equation}
For $u\in H^2$,
$|\nabla\mathcal N(u)|\le C(|u||\nabla^2u|+|\nabla u|^2)$,
$\|\mathcal N(u)\|_{H^{-2}}\le C\|u\|_2(\|u\|_{H^2}+\|\Div u\|_2)$, and
\begin{equation}
\label{eq:N-H1-estimate}
\|\mathcal{N}(u)\|_{H^1}
\le
C\|u\|_{H^2}^2,
\end{equation}
\begin{equation}
\label{eq:N-H1-lipschitz}
\|\mathcal{N}(u)-\mathcal{N}(v)\|_{H^1}
\le
C\bigl(\|u\|_{H^2}+\|v\|_{H^2}\bigr)\|u-v\|_{H^2}.
\end{equation}
\end{lem}

\begin{proof}
Pointwise, we have
\[
|(u\cdot\nabla)u|\le |u||\nabla u|,
\qquad
|\Div u|\le\sqrt3|\nabla u|,
\]
which yield the first two bounds. Since $u\otimes u\in W^{1,1}$, the
integral of $\Div(u\otimes u)$ vanishes. Combining this observation
with~\eqref{eq:N-decomp} gives the third identity.

To prove~\eqref{eq:Nhat-low-frequency}, we first observe that the
Fourier transform of $\Div(u\otimes u)$ satisfies
\[
|\widehat{\Div(u\otimes u)}(\xi)|
\le |\xi|\,|\widehat{u\otimes u}(\xi)|
\le (2\pi)^{-3/2}|\xi|\int|u|^2.
\]
Moreover, by the Cauchy--Schwarz inequality,
\[
|\widehat{u\Div u}(\xi)|
\le (2\pi)^{-3/2}\|u\|_2\|\Div u\|_2.
\]
The pointwise bound on $\nabla\mathcal N(u)$ follows directly.

For the $H^{-2}$ estimate, we use~\eqref{eq:N-decomp} together with
\[
\|\Div(u\otimes u)\|_{H^{-2}}
\le\|u\otimes u\|_2
\le\|u\|_2\|u\|_\infty,
\]
whereas the embedding $H^2\hookrightarrow L^\infty$ and duality give
\[
\|u\,\Div u\|_{H^{-2}}
\le C\|u\,\Div u\|_1.
\]
Finally, \eqref{eq:N-H1-estimate} and~\eqref{eq:N-H1-lipschitz}
follow from the product estimate $H^2\times H^1\hookrightarrow H^1$
in three dimensions, since $\mathcal N$ is bilinear in $(u,\nabla u)$.
\end{proof}
\begin{rem}\rm
\label{rem:model-scope}
The analysis in this paper is carried out for~\eqref{Main-eq}. A
formulation that retains a pressure term $\nabla p$ alongside the
penalization would require additional information, such as an equation
of state or a divergence constraint, to determine $p$. The pressure
associated with the limiting problem~\eqref{eq:HNS} is recovered in
Theorem~\ref{thm:incompressible-limit}(iii).
\end{rem}

\subsection{The linear semigroup}
\label{subsec:semigroup}

We write~\eqref{Main-eq} in the form
$\partial_tu=L_\varepsilon u-\mathcal N(u)$, where
\begin{equation}
\label{eq:Leps-correct}
L_\varepsilon
:=
\nu\Delta-\beta\Delta^2+\frac1\varepsilon\nabla\Div ,
\end{equation}
The semigroup $S_\varepsilon(t):=e^{tL_\varepsilon}$ acts on the Fourier side by
\begin{equation}
\label{eq:semigroup-fourier}
\widehat{S_\varepsilon(t)f}(\xi)
=
e^{-tM_\varepsilon(\xi)}\widehat f(\xi),
\end{equation}
where
\begin{equation}
\label{eq:Meps}
M_\varepsilon(\xi)
=
a(\xi)I
+\frac1\varepsilon \xi\otimes\xi,
\qquad a(\xi):=\nu|\xi|^2+\beta|\xi|^4 .
\end{equation}
Let $\widehat{\Pp}(\xi)=I-\xi\otimes\xi/|\xi|^2$ and
$\widehat{\Qq}(\xi)=\xi\otimes\xi/|\xi|^2$ be the symbols of the Leray
projector and of the gradient projector. They are complementary
orthogonal projections in $\R^3$ for each $\xi\ne0$, and
$M_\varepsilon(\xi)=a(\xi)\widehat\Pp(\xi)
+\bigl(a(\xi)+\varepsilon^{-1}|\xi|^2\bigr)\widehat\Qq(\xi)$. Hence
\begin{equation}
\label{eq:semigroup-split}
e^{-tM_\varepsilon(\xi)}
=e^{-ta(\xi)}\widehat\Pp(\xi)+e^{-t(a(\xi)+|\xi|^2/\varepsilon)}\widehat\Qq(\xi),
\qquad
S_\varepsilon(t)=e^{t(\nu\Delta-\beta\Delta^2)}\Pp
+e^{t((\nu+\varepsilon^{-1})\Delta-\beta\Delta^2)}\Qq .
\end{equation}
Thus the penalization acts only on gradient fields, as a bulk viscosity,
and the solenoidal part $\Pp S_\varepsilon(t)=e^{t(\nu\Delta-\beta\Delta^2)}\Pp$
does not depend on $\varepsilon$. In particular
$\|e^{-tM_\varepsilon(\xi)}\|_{\mathcal L(\R^3)}=e^{-ta(\xi)}
\le\min\{e^{-\nu t|\xi|^2},e^{-\beta t|\xi|^4}\}$, so that
\begin{equation}
\label{eq:semigroup-L2}
\|S_\varepsilon(t)f\|_2\le \|f\|_2,
\qquad
\|S_\varepsilon(t)f\|_2\le(8\pi\nu t)^{-3/4}\|f\|_1,
\end{equation}
and, for each integer $m\ge0$ and $t>0$,
\begin{equation}
\label{eq:semigroup-smoothing}
\|\nabla^m S_\varepsilon(t)f\|_2
\le
C_m (\nu t)^{-m/2}\|f\|_2,
\qquad
\|S_\varepsilon(t)f\|_{H^2}\le C\bigl(1+(\beta t)^{-1/4}\bigr)\|f\|_{H^1}.
\end{equation}
The second bound in~\eqref{eq:semigroup-L2} follows from the bound
$|\widehat f|\le(2\pi)^{-3/2}\|f\|_1$ and the identity
$\int e^{-2\nu t|\xi|^2}d\xi=(\pi/(2\nu t))^{3/2}$. All these constants
are independent of $\varepsilon$.

\begin{rem}\rm
\label{rem:semigroup-biharmonic}
The lower bound $a(\xi)\ge\beta|\xi|^4$ gives the biharmonic smoothing
estimate
$\|\nabla^m S_\varepsilon(t)f\|_2\le C_m (\beta t)^{-m/4}\|f\|_2$. It is
stronger for short times and it gives the second inequality
in~\eqref{eq:semigroup-smoothing}. For large times the heat-type bounds
are the relevant ones, because $\beta|\xi|^4$ is negligible with respect
to $\nu|\xi|^2$ at low frequencies.
\end{rem}

For $m\in\R^3$ and $t>0$, we define $K_\varepsilon(t)m\in L^2(\R^3;\R^3)$
by
\begin{equation}
\label{eq:Keps}
\widehat{K_\varepsilon(t)m}(\xi):=(2\pi)^{-3/2}e^{-tM_\varepsilon(\xi)}m ,
\end{equation}
which is $S_\varepsilon(t)$ applied to the point momentum $m\,\boldsymbol\delta$.

\begin{lem}
\label{lem:kernel}
For $m\in\R^3$ and $t>0$,
$\|K_\varepsilon(t)m\|_2\le(8\pi\nu t)^{-3/4}|m|$ and
$$
\lim_{t\to\infty}t^{3/2}\|K_\varepsilon(t)m\|_2^2=\gamma_\varepsilon\,(8\pi\nu)^{-3/2}|m|^2 ,
$$
with $\gamma_\varepsilon$ given by~\eqref{eq:gamma-eps}.
\end{lem}

\begin{proof}
The first bound follows from~\eqref{eq:semigroup-L2}. By
\eqref{eq:semigroup-split} and the orthogonality of
$\widehat\Pp(\xi)m$ and $\widehat\Qq(\xi)m$,
$$
\begin{aligned}
t^{3/2}\|K_\varepsilon(t)m\|_2^2
&=(2\pi)^{-3}\int_{\R^3}e^{-2(\nu|\eta|^2+\beta|\eta|^4/t)}\\
&\qquad\times\Bigl(|\widehat\Pp(\eta)m|^2+e^{-2|\eta|^2/\varepsilon}|\widehat\Qq(\eta)m|^2\Bigr)d\eta ,
\end{aligned}
$$
after the change of variables $\xi=\eta/\sqrt t$. By monotone
convergence, as $t\to\infty$, this tends to the same integral without the
term $\beta|\eta|^4/t$. In polar coordinates, since
$\int_{S^2}(\omega\cdot m)^2d\sigma(\omega)=\frac{4\pi}{3}|m|^2$, the
angular averages of $|\widehat\Pp(\omega)m|^2$ and
$|\widehat\Qq(\omega)m|^2$ are $\frac23|m|^2$ and $\frac13|m|^2$. Using
$(2\pi)^{-3}\int e^{-2c|\eta|^2}d\eta=(8\pi c)^{-3/2}$ with $c=\nu$ and
$c=\nu+\varepsilon^{-1}$, we find that the limit equals
$$
(8\pi\nu)^{-3/2}|m|^2\bigl[\frac23+\frac13(\nu/(\nu+\varepsilon^{-1}))^{3/2}\bigr].
$$
\end{proof}

\section{Weak solutions: proof of Theorem~\ref{thm:wp}}
\label{sec:existence}

\subsection{Existence}

For $n\in\N$, let $J_n$ be the Fourier multiplier with symbol
$\mathbf{1}_{\{|\xi|\le n\}}$, the orthogonal projection of $L^2$ onto
$X_n:=\{f\in L^2(\R^3;\R^3):\operatorname{supp}\widehat{f}\subset
\overline{B(0,n)}\}$. We consider the approximate problem
\begin{equation}
    \label{eq:approx}
    \partial_t u_n
    - \nu\Delta u_n
    + \beta\Delta^2u_n
    - \frac1\varepsilon \nabla\Div u_n
    + J_n\mathcal{N}(u_n)=0,
    \qquad
    u_n(0)=J_nu_0.
\end{equation}
The operators $\Delta$, $\Delta^2$ and $\nabla\Div$ map $X_n$ into
itself and are bounded on it, and the same holds for $J_n\mathcal N$. On
$X_n$, all Sobolev norms are equivalent by Bernstein's inequality, so that
$\|J_n\mathcal N(v)-J_n\mathcal N(w)\|_2\le
C_n(\|v\|_2+\|w\|_2)\|v-w\|_2$ for $v,w\in X_n$. The Picard theorem in
the Banach space $X_n$ yields a unique maximal solution
$u_n\in C^1([0,T_n);X_n)$.

Taking the inner product of~\eqref{eq:approx} with $u_n$ and using that
$J_n$ is self-adjoint with $J_nu_n=u_n$, we get
$\int J_n\mathcal N(u_n)\cdot u_n\,dx=\int\mathcal N(u_n)\cdot u_n\,dx=0$
by Lemma~\ref{lem:cancel}. Hence
\begin{equation}
    \label{eq:energy-approx}
    \frac12\frac{d}{dt}\|u_n\|_{2}^2
    + \nu\|\nabla u_n\|_{2}^2
    + \beta\|\Delta u_n\|_{2}^2
    + \frac1\varepsilon\|\Div u_n\|_{2}^2
    =0,
\end{equation}
and, since $\|J_nu_0\|_2\le\|u_0\|_2$,
\begin{equation}
    \label{eq:uniform-energy}
    \|u_n(t)\|_{2}^2
    + 2\int_0^t\Bigl(\nu\|\nabla u_n\|_{2}^2
    + \beta\|\Delta u_n\|_{2}^2
    + \frac{1}{\varepsilon}\|\Div u_n\|_{2}^2\Bigr)d\tau
    \le \|u_0\|_{2}^2
\end{equation}
on $[0,T_n)$. In particular, $\|u_n(t)\|_2$ remains bounded, and hence
$T_n=\infty$.
Since $\|u_n\|_{H^2}^2\le 2(\|u_n\|_2^2+\|\Delta u_n\|_2^2)$, the
sequence $(u_n)$ is bounded in $L^\infty(0,T;L^2)\cap L^2(0,T;H^2)$ for
every $T>0$. By Remark~\ref{rem:weak-solution-comments}(b),
\begin{equation}
\label{eq:L4bound}
\int_0^T\|u_n\|_4^4\,d\tau\le C\sup_{[0,T]}\|u_n\|_2^2\int_0^T\|u_n\|_{H^2}^2\,d\tau
\le C\|u_0\|_2^4\Bigl(T+\frac1\beta\Bigr).
\end{equation}

Next, $\partial_tu_n=\nu\Delta u_n-\beta\Delta^2u_n
+\varepsilon^{-1}\nabla\Div u_n-J_n\mathcal N(u_n)$. The linear terms are
bounded in $H^{-2}$ by $C(\|u_n\|_2+\|\Delta u_n\|_2
+\varepsilon^{-1}\|\Div u_n\|_2)$, and $J_n$ is a contraction on
$H^{-2}$. By Lemma~\ref{lem:N-bounds},
$\|J_n\mathcal N(u_n)\|_{H^{-2}}\le C\|u_n\|_2(\|u_n\|_{H^2}+\|\Div u_n\|_2)$.
By~\eqref{eq:uniform-energy} we conclude that, for fixed $\varepsilon$,
\begin{equation}
    \label{eq:dtbound}
    \{\partial_tu_n\}_{n\ge1}
    \quad\text{is bounded in}\quad
    L^{2}(0,T;H^{-2}(\R^3)).
\end{equation}
The bound depends on $\varepsilon$ through the penalization term, which
is harmless here because $\varepsilon$ is fixed in this section.

Fix $R>0$. The restrictions of $u_n$ to the ball $B_R$ are bounded in
$L^2(0,T;H^2(B_R))$ and their time derivatives are bounded in
$L^2(0,T;H^{-2}(B_R))$, where $H^{-2}(B_R)$ is the dual of
$H^2_0(B_R)$. Since $H^2(B_R)\Subset H^1(B_R)\hookrightarrow H^{-2}(B_R)$,
the Aubin--Lions--Simon lemma~\cite{Simon1987} and a diagonal argument
in $R$ give a subsequence, still denoted $(u_n)$, and a limit $u$ such
that
\begin{align}
    u_n &\to u
    &&\text{strongly in }L^2(0,T;H^1(B_R))\ \text{for every }R>0,
    \label{eq:strong-loc}\\
    u_n &\rightharpoonup u
    &&\text{weakly in }L^2(0,T;H^2(\R^3)),
    \label{eq:weak-H2}\\
    u_n &\overset{*}{\rightharpoonup} u
    &&\text{weak-* in }L^\infty(0,T;L^2(\R^3)).
    \label{eq:weak-star}
\end{align}

We now pass to the limit in the nonlinear term. Let
$\phi\in C_c^\infty((0,T)\times\R^3;\R^3)$. By~\eqref{eq:N-decomp},
$\langle\mathcal N(u_n),\phi\rangle
=-\iint(u_n\otimes u_n):\nabla\phi
-\frac12\iint(\Div u_n)(u_n\cdot\phi)$. The strong convergence
\eqref{eq:strong-loc} gives $u_n\otimes u_n\to u\otimes u$ in
$L^1_{\mathrm{loc}}$ and $u_n\cdot\phi\to u\cdot\phi$ in $L^2$, while
$\Div u_n\rightharpoonup\Div u$ weakly in $L^2((0,T)\times\R^3)$; hence
$\langle\mathcal N(u_n),\phi\rangle\to\langle\mathcal N(u),\phi\rangle$.
To remove $J_n$ we write
$\langle J_n\mathcal N(u_n)-\mathcal N(u),\phi\rangle
=\langle\mathcal N(u_n),J_n\phi-\phi\rangle
+\langle\mathcal N(u_n)-\mathcal N(u),\phi\rangle$; the first term tends
to zero because $J_n\phi\to\phi$ in $L^2(0,T;H^2)$ and
$\mathcal N(u_n)$ is bounded in $L^2(0,T;H^{-2})$. Multiplying
\eqref{eq:approx} by $\phi\in C_c^\infty([0,T)\times\R^3;\R^3)$,
integrating by parts, and using
\eqref{eq:strong-loc}--\eqref{eq:weak-star} and $J_nu_0\to u_0$ in $L^2$,
we obtain~\eqref{eq:weak-form}.

Finally, a further extraction in~\eqref{eq:strong-loc} gives
$u_n(t)\to u(t)$ in $L^2_{\mathrm{loc}}$ for almost every $t$, hence
$u_n(t)\rightharpoonup u(t)$ weakly in $L^2$ by~\eqref{eq:uniform-energy}.
Weak lower semicontinuity of the norms, together with the weak
convergence of $\nabla u_n$, $\Delta u_n$ and $\Div u_n$ in
$L^2((0,t)\times\R^3)$, gives~\eqref{eq:energy-ineq} for almost every
$t$. Thus $u$ is a global weak solution, and $u\in L^4(0,T;L^4)$
by~\eqref{eq:L4bound} and lower semicontinuity. This completes the
proof of existence.

\begin{rem}\rm
\label{rem:no-leray}
The pure Fourier cut-off $J_n$ in~\eqref{eq:approx} cannot be replaced
by $J_n\Pp$. Otherwise every approximate solution would be
divergence-free, the penalty term would vanish identically, and the limit
problem would be the incompressible projected system instead
of~\eqref{Main-eq}.
\end{rem}

\subsection{Uniqueness, continuity and energy equality}

\begin{prop}
\label{prop:weak-unique}
Let $u$ be a global weak solution with data $u_0\in L^2$. Then
$u\in C([0,\infty);L^2)$, $\partial_tu\in L^2(0,T;H^{-2})$ for every
$T>0$, and the energy equality~\eqref{eq:energy-equality} holds for every
$t\ge0$. If $v$ is a global weak solution with data $v_0$, then
$w:=u-v$ satisfies, for $t\ge0$,
\begin{equation}
\label{eq:stability}
\|w(t)\|_2^2\le\|u_0-v_0\|_2^2\exp\Bigl(\int_0^t g(s)\,ds\Bigr),
\qquad
g:=C\beta^{-3/5}\bigl(\|\nabla u\|_2+\|\nabla v\|_2\bigr)^{8/5}+\frac{1}{4\nu}\|v\|_\infty^2,
\end{equation}
where $C$ is absolute and $g\in L^1_{\mathrm{loc}}([0,\infty))$.
\end{prop}

\begin{proof}
Let $F:=\nu\Delta u-\beta\Delta^2u+\varepsilon^{-1}\nabla\Div u-\mathcal N(u)$.
By item (i) of Definition~\ref{def:weak-sol} and Lemma~\ref{lem:N-bounds},
$F\in L^2(0,T;H^{-2})$. Taking $\phi(t,x)=\theta(t)\psi(x)$ in
\eqref{eq:weak-form} with $\theta\in C_c^\infty((0,T))$ and
$\psi\in C_c^\infty(\R^3)$ shows that $\partial_tu=F$ in
$\mathcal D'(0,T;H^{-2})$, by density of $C_c^\infty$ in $H^2$. Hence
$u\in C([0,T];H^{-2})$, and taking now $\theta\in C_c^\infty([0,T))$
with $\theta(0)=1$ and integrating by parts in time gives $u(0)=u_0$. Since
$u\in L^2(0,T;H^2)$ and $\partial_tu\in L^2(0,T;H^{-2})$, the
Lions--Magenes lemma (see~\cite[Ch.~III, Lemma~1.2]{Temam}) gives
$u\in C([0,T];L^2)$ and $\frac{d}{dt}\|u\|_2^2=2\langle\partial_tu,u\rangle$
almost everywhere. The duality $H^{-2}\times H^2$ coincides with the
integral $\int\mathcal N(u)\cdot u\,dx$ because $\mathcal N(u)\in L^1$
and $u\in L^\infty$, so Lemma~\ref{lem:cancel}
gives~\eqref{eq:energy-equality}.

The difference $w=u-v$ satisfies, in the same way,
$$
\frac12\frac{d}{dt}\|w\|_2^2+\nu\|\nabla w\|_2^2+\beta\|\Delta w\|_2^2
+\frac1\varepsilon\|\Div w\|_2^2
=-\int_{\R^3}\bigl(\mathcal N(u)-\mathcal N(v)\bigr)\cdot w\,dx .
$$
Let $B(a,b):=(a\cdot\nabla)b+\frac12a\,\Div b$, so that
$\mathcal N(u)-\mathcal N(v)=B(w,u)+B(v,w)$. First,
$|\int B(w,u)\cdot w|\le c_N\|w\|_4^2\|\nabla u\|_2$. Second, integrating
by parts in the transport term,
$$
\int B(v,w)\cdot w\,dx=-\frac12\int(\Div v)|w|^2\,dx+\frac12\int(\Div w)(v\cdot w)\,dx ,
$$
and the two terms are bounded by $\frac12\|\nabla v\|_2\|w\|_4^2$ and
$\frac12\|\nabla w\|_2\|v\|_\infty\|w\|_2$, since
$\|\Div f\|_2\le\|\nabla f\|_2$. By~\eqref{eq:ineq-agmon},
$\|w\|_4^2\le\|w\|_2\|w\|_\infty\le C\|w\|_2^{5/4}\|\Delta w\|_2^{3/4}$,
so Young's inequality with exponents $8/3$ and $8/5$ gives
$$
\Bigl|\int(\mathcal N(u)-\mathcal N(v))\cdot w\,dx\Bigr|
\le\frac\beta2\|\Delta w\|_2^2+\frac\nu2\|\nabla w\|_2^2
+\frac12g(t)\|w\|_2^2 .
$$
Gronwall's inequality yields~\eqref{eq:stability}. Since $\nabla u,\nabla
v\in L^2(0,T;L^2)$ and $\|v\|_\infty\le C\|v\|_{H^2}\in L^2(0,T)$, the
function $g$ is locally integrable.
\end{proof}

\begin{proof}[Proof of Theorem~\ref{thm:wp}]
Existence was proved above, and Proposition~\ref{prop:weak-unique} gives
uniqueness, continuity, the energy equality and the stability estimate.
By~\eqref{eq:energy-equality}, $\int_0^T(\|\nabla u\|_2^2+\|\nabla
v\|_2^2)\le(\|u_0\|_2^2+\|v_0\|_2^2)/(2\nu)$ and
$\int_0^T\|v\|_{H^2}^2\le C(T+\beta^{-1})\|v_0\|_2^2$, which bounds
$\int_0^Tg$ independently of $\varepsilon$.
\end{proof}

\section[Global strong solutions: proof of Theorem~\ref{thm:uniform-small-data-strong}]{Global strong solutions:\ proof of Theorem~\ref{thm:uniform-small-data-strong}}
\label{sec:strong-solutions}

\subsection{Local strong solutions}

By a strong solution on $[0,T]$, we mean a function
$u\in C([0,T];H^2)\cap L^2(0,T;H^4)$ satisfying~\eqref{Main-eq} in
$L^2(0,T;L^2)$. A strong solution is a weak solution, so it is unique by
Proposition~\ref{prop:weak-unique}.

\begin{lem}
\label{lem:local}
For every $u_0\in H^2$ there is $T>0$, depending only on $\beta$ and
$\|u_0\|_{H^2}$, and a strong solution on $[0,T]$. It satisfies the mild
formulation
\begin{equation}
\label{eq:mild-strong}
u(t)
=
S_\varepsilon(t)u_0
-
\int_0^t S_\varepsilon(t-s)\mathcal{N}(u(s))\,ds .
\end{equation}
Consequently, there exist a maximal time $T_{\max}\in(0,\infty]$ and a
unique strong solution on every $[0,T]\subset[0,T_{\max})$, and if
$T_{\max}<\infty$, then
\begin{equation}
\label{eq:continuation-criterion}
\limsup_{t\uparrow T_{\max}}\|u(t)\|_{H^2}=+\infty .
\end{equation}
\end{lem}

\begin{proof}
By~\eqref{eq:semigroup-smoothing}, \eqref{eq:N-H1-estimate} and~\eqref{eq:N-H1-lipschitz}, the map
defined by the right side of~\eqref{eq:mild-strong} satisfies
$$
\Bigl\|\int_0^tS_\varepsilon(t-s)\bigl(\mathcal N(u)-\mathcal N(v)\bigr)ds\Bigr\|_{H^2}
\le C\bigl(T+\beta^{-1/4}T^{3/4}\bigr)\sup_{[0,T]}\bigl(\|u\|_{H^2}+\|v\|_{H^2}\bigr)\sup_{[0,T]}\|u-v\|_{H^2},
$$
so it is a contraction on a closed ball of $C([0,T];H^2)$ of radius
$2\|u_0\|_{H^2}$ when $T$ is small in terms of $\beta$ and
$\|u_0\|_{H^2}$, uniformly in $\varepsilon$. For the $H^4$ regularity,
let $f:=\mathcal N(u)\in L^\infty(0,T;H^1)$. Since
$\|e^{-\tau M_\varepsilon(\xi)}\|\le e^{-\beta\tau|\xi|^4}$, Young's
inequality in time gives, for each $\xi$,
$$
\Bigl\||\xi|^4\int_0^te^{-(t-s)M_\varepsilon(\xi)}\widehat f(s,\xi)\,ds\Bigr\|_{L^2(0,T)}
\le\beta^{-1}\|\widehat f(\cdot,\xi)\|_{L^2(0,T)},
$$
and
$\int_0^\infty|\xi|^8e^{-2\beta t|\xi|^4}dt\,|\widehat{u_0}|^2
=(2\beta)^{-1}|\xi|^4|\widehat{u_0}|^2$. Integrating in $\xi$, we obtain
$$
\|\Delta^2u\|_{L^2(0,T;L^2)}\le(2\beta)^{-1/2}\|\Delta u_0\|_2+\beta^{-1}\|f\|_{L^2(0,T;L^2)},
$$
so $u\in L^2(0,T;H^4)$. Since the mild solution satisfies
$\partial_tu=L_\varepsilon u-f$ in the sense of distributions, we also
get $\partial_tu\in L^2(0,T;L^2)$, and $u$ is a strong solution.
\end{proof}

\subsection{Higher-order energy inequalities}

For a strong solution $u$ and $k=0,1,2$, we set
$$
E_k(t):=\|\nabla^ku(t)\|_2^2,\qquad D_k(t):=\|\nabla^{k+1}u(t)\|_2^2 .
$$
Since $u\in L^2(0,T;H^4)$ and $\partial_tu\in L^2(0,T;L^2)$, the
functions $E_1$ and $E_2$ are absolutely continuous, with
$E_1'=2\langle\partial_tu,-\Delta u\rangle$ and
$E_2'=2\langle\partial_tu,\Delta^2u\rangle$ (Lions--Magenes lemma). Using
\eqref{Main-eq}, we obtain for almost every $t$
\begin{equation}
\label{eq:H2-diff-alpha}
\frac12E_k'+\nu D_k+\beta\|\nabla^{k+2}u\|_2^2+\frac1\varepsilon\|\nabla^k\Div u\|_2^2
=R_k,
\end{equation}
where $R_0=0$ by Lemma~\ref{lem:cancel}, and
\begin{equation}
\label{eq:E1-energy}
R_1=\int_{\R^3}\mathcal N(u)\cdot\Delta u\,dx,
\end{equation}
\begin{equation}
\label{eq:E2-energy}
R_2=-\langle\mathcal N(u),\Delta^2u\rangle=\int_{\R^3}\nabla\mathcal N(u):\nabla\Delta u\,dx .
\end{equation} Here
the penalization terms were computed on the Fourier side; for instance
$$
\langle-\varepsilon^{-1}\nabla\Div u,\Delta^2u\rangle
=\varepsilon^{-1}\int|\xi|^4|\xi\cdot\widehat u|^2\,d\xi
=\varepsilon^{-1}\|\nabla^2\Div u\|_2^2 .
$$

\begin{lem}
\label{lem:Hk-key}
There is an absolute constant $C_*$ such that, for $u\in H^3$,
\begin{equation}
\label{eq:E1-nonlinear}
\Bigl|\int_{\R^3}\mathcal N(u)\cdot\Delta u\,dx\Bigr|\le C_*\|u\|_{\dot H^{1/2}}\|\nabla^2u\|_2^2,
\end{equation}
\begin{equation}
\label{eq:E2-nonlinear}
\Bigl|\int_{\R^3}\nabla\mathcal N(u):\nabla\Delta u\,dx\Bigr|\le C_*\|u\|_{\dot H^{1/2}}\|\nabla^3u\|_2^2 .
\end{equation}
\end{lem}

\begin{proof}
By~\eqref{eq:N-pointwise}, H\"older's inequality and~\eqref{eq:ineq-sobolev},
$$
\int|\mathcal N(u)||\Delta u|\le c_N\|u\|_3\|\nabla u\|_6\|\Delta u\|_2
\le C\|u\|_{\dot H^{1/2}}\|\nabla^2u\|_2^2 .
$$
For the second estimate, we use~\eqref{eq:ineq-sobolev}
and~\eqref{eq:ineq-interp} to get
$$
\|\nabla\mathcal N(u)\|_2\le C(\|u\|_3\|\nabla^2u\|_6+\|\nabla u\|_4^2)
\le C\|u\|_{\dot H^{1/2}}\|\nabla^3u\|_2 ,
$$
and we recall that $\|\nabla\Delta u\|_2=\|\nabla^3u\|_2$.
\end{proof}

Combining~\eqref{eq:H2-diff-alpha}--\eqref{eq:E2-energy} with~\eqref{eq:E1-nonlinear}--\eqref{eq:E2-nonlinear} and
$\|u\|_{\dot H^{1/2}}^2\le\|u\|_2\|\nabla u\|_2=(E_0E_1)^{1/2}$, we
obtain for $k=1,2$
\begin{equation}
\label{eq:H2-nonlinear-key}
\frac12E_k'+\nu D_k+\beta\|\nabla^{k+2}u\|_2^2+\frac1\varepsilon\|\nabla^k\Div u\|_2^2
\le C_*(E_0E_1)^{1/4}D_k .
\end{equation}
We use this inequality in two ways. First, using the biharmonic
dissipation and the bound $D_k\le E_k^{1/2}\|\nabla^{k+2}u\|_2$
(by~\eqref{eq:ineq-interp}), we obtain by Young's inequality
\begin{equation}
\label{eq:H2-main-energy}
E_k'+2\nu D_k+\beta\|\nabla^{k+2}u\|_2^2+\frac2\varepsilon\|\nabla^k\Div u\|_2^2
\le\frac{C_*^2}{\beta}\,\|u\|_2\|\nabla u\|_2\,E_k .
\end{equation}
Second, using the viscous dissipation instead, we set
\begin{equation}
\label{eq:delta-choice}
\delta_0:=\frac{\nu^2}{16C_*^2}.
\end{equation}
On every interval where $E_0E_1\le16\delta_0^2$, we have
$C_*(E_0E_1)^{1/4}\le\nu/2$, and~\eqref{eq:H2-nonlinear-key} gives, for $k=1,2$,
\begin{equation}
\label{eq:H2-absorbed}
E_k'+\nu D_k+2\beta\|\nabla^{k+2}u\|_2^2+\frac2\varepsilon\|\nabla^k\Div u\|_2^2\le0 .
\end{equation}

\subsection{Proof of Theorem~\ref{thm:uniform-small-data-strong}}

Let $u$ be the strong solution given by Lemma~\ref{lem:local}, defined
on $[0,T_{\max})$. It coincides with the weak solution of
Theorem~\ref{thm:wp}. We first consider small data. Assume
$E_0(0)E_1(0)\le\delta_0^2$, and let $T'$ be the supremum of the times
$T<T_{\max}$ such that $E_0E_1\le16\delta_0^2$ on $[0,T]$. By
continuity, $T'>0$. On $[0,T')$, the function $E_0$ is nonincreasing
by~\eqref{eq:energy-equality} and $E_1$ is nonincreasing
by~\eqref{eq:H2-absorbed}. Hence $E_0E_1\le\delta_0^2$ on $[0,T')$, and
therefore $T'=T_{\max}$. Thus~\eqref{eq:H2-absorbed} holds on
$[0,T_{\max})$ for $k=1,2$. We add the
energy equality and~\eqref{eq:H2-absorbed} for $k=1,2$, and we use the
identities
$$
E_1+D_1+D_2=\|\nabla u\|_{H^2}^2,\qquad
\sum_{k\le2}\|\nabla^{k+2}u\|_2^2=\|\Delta u\|_{H^2}^2,\qquad
\sum_{k\le2}\|\nabla^k\Div u\|_2^2=\|\Div u\|_{H^2}^2 .
$$
This yields
\begin{equation}
\label{eq:H2-final-estimate-finite}
\|u(t)\|_{H^2}^2
+\int_0^t\Bigl(\nu\|\nabla u\|_{H^2}^2+2\beta\|\Delta u\|_{H^2}^2
+\frac2\varepsilon\|\Div u\|_{H^2}^2\Bigr)d\tau
\le\|u_0\|_{H^2}^2,
\qquad t<T_{\max}.
\end{equation}
By~\eqref{eq:continuation-criterion}, $T_{\max}=\infty$,
and~\eqref{eq:uniform-H2-bound} follows by letting $t\to\infty$. This
proves the second part of the theorem with $c_0=(16C_*^2)^{-1}$.

For general data, we first note that~\eqref{eq:H2-main-energy} and
Gronwall's inequality give, for $k=1,2$ and $t<T_{\max}$,
\begin{equation}
\label{eq:large-Gronwall}
E_k(t)+\int_0^t\Bigl(2\nu D_k+\beta\|\nabla^{k+2}u\|_2^2+\frac2\varepsilon\|\nabla^k\Div u\|_2^2\Bigr)d\tau
\le E_k(0)\exp\Bigl(\frac{C_*^2}{\beta}\int_0^t\|u\|_2\|\nabla u\|_2\,d\tau\Bigr),
\end{equation}
and $\int_0^t\|u\|_2\|\nabla u\|_2\le\|u_0\|_2^2\,(t/2\nu)^{1/2}$
by~\eqref{eq:energy-equality}. Hence $\|u(t)\|_{H^2}$ is bounded on
bounded subsets of $[0,T_{\max})$, and $T_{\max}=\infty$
by~\eqref{eq:continuation-criterion}. To obtain bounds that are uniform
in time, we set $T_*:=\|u_0\|_2^4/(2\nu\delta_0^2)$. By~\eqref{eq:energy-equality},
$\int_0^{T_*}E_0E_1\,dt\le\|u_0\|_2^2\int_0^\infty E_1\,dt\le\|u_0\|_2^4/(2\nu)=T_*\delta_0^2$,
so by continuity there is $t_0\in[0,T_*]$ with
$E_0(t_0)E_1(t_0)\le\delta_0^2$. On $[0,t_0]$, we
use~\eqref{eq:large-Gronwall}, in which the exponential factor is at most
$$
\Lambda_0:=\exp\Bigl(\frac{C_*^2\|u_0\|_2^2}{\beta}\Bigl(\frac{T_*}{2\nu}\Bigr)^{1/2}\Bigr)
=\exp\Bigl(\frac{C_*^2\|u_0\|_2^4}{2\beta\nu\delta_0}\Bigr)
=\exp\Bigl(\frac{8C_*^4\|u_0\|_2^4}{\beta\nu^3}\Bigr).
$$
On $[t_0,\infty)$, we apply the small-data case to $u(t_0+\cdot)$, which
is the solution with data $u(t_0)$ by uniqueness. Hence $E_1$ and $E_2$
are nonincreasing there, and the integrals of their dissipations are
bounded by $E_k(t_0)\le\Lambda_0E_k(0)$. Together with the energy equality for
$k=0$, this gives
\begin{equation}
\label{eq:uniform-large}
\sup_{t\ge0}E_k(t)\le\Lambda_0E_k(0),
\qquad
\int_0^\infty\Bigl(\nu D_k+\beta\|\nabla^{k+2}u\|_2^2+\frac1\varepsilon\|\nabla^k\Div u\|_2^2\Bigr)dt\le2\Lambda_0E_k(0),
\end{equation}
for $k=1,2$. Summing over $k$ and adding the energy equality, we
obtain~\eqref{eq:large-data-bound} with $c=8C_*^4$. Finally,
$\int_0^T\|\Delta u\|_{H^2}^2<\infty$ gives
$u\in L^2_{\mathrm{loc}}([0,\infty);H^4)$. This completes the proof of
Theorem~\ref{thm:uniform-small-data-strong}. \qed

\section{Decay: proof of Theorem~\ref{thm:optimal-decay}}
\label{sec:optimal-decay}

Throughout this section, $u$ denotes the weak solution with data
$u_0\in L^1\cap L^2$, and we write $E_0(t):=\|u(t)\|_2^2$. By the energy
equality,
\begin{equation}
\label{eq:energy-decay-section}
\frac{d}{dt}E_0
+2\nu\|\nabla u\|_2^2
+2\beta\|\Delta u\|_2^2
+\frac2\varepsilon\|\Div u\|_2^2
=0
\quad\text{a.e.},
\end{equation}
hence
\begin{equation}
\label{eq:basic-integral-bounds}
\sup_{t\ge0}E_0(t)+2\int_0^\infty\Bigl(\nu\|\nabla u\|_2^2+\beta\|\Delta u\|_2^2+\frac1\varepsilon\|\Div u\|_2^2\Bigr)dt\le\|u_0\|_2^2,
\end{equation}
and, since $\|\Div u\|_2\le\|\nabla u\|_2$,
\begin{equation}
\label{eq:div-uniform-bound}
\int_0^\infty\|\Div u(t)\|_2^2\,dt
\le\kappa_\varepsilon^2\|u_0\|_2^2,
\qquad
\kappa_\varepsilon^2=\frac12\min\{\varepsilon,\nu^{-1}\}.
\end{equation}
Note that the bound $\frac\varepsilon2\|u_0\|_2^2$ alone is not uniform
for large $\varepsilon$; taking the minimum with $\nu^{-1}$ makes it
uniform.

\subsection{The Duhamel formula on the Fourier side}

\begin{lem}
\label{lem:duhamel-weak}
For almost every $\xi\in\R^3$ and every $t\ge0$,
\begin{equation}
\label{eq:uhat-basic}
\widehat u(t,\xi)=e^{-tM_\varepsilon(\xi)}\widehat{u_0}(\xi)
-\int_0^te^{-(t-s)M_\varepsilon(\xi)}\widehat{\mathcal N(u)}(s,\xi)\,ds .
\end{equation}
Consequently, with $\Phi(t):=\int_0^tE_0(s)\,ds$,
\begin{equation}
\label{eq:Nhat-estimate}
|\widehat u(t,\xi)|
\le(2\pi)^{-3/2}\Bigl(\|u_0\|_1+|\xi|\,\Phi(t)
+\frac12\kappa_\varepsilon\|u_0\|_2\,\Phi(t)^{1/2}\Bigr).
\end{equation}
\end{lem}

\begin{proof}
By Proposition~\ref{prop:weak-unique}, $\partial_tu=L_\varepsilon u-\mathcal N(u)$
in $L^2(0,T;H^{-2})$. Taking Fourier transforms,
$\partial_t\widehat u=-M_\varepsilon\widehat u-\widehat{\mathcal N(u)}$
in $L^2(0,T;L^2((1+|\xi|^2)^{-2}d\xi))$, so by Fubini's theorem, for
almost every $\xi$, the function $t\mapsto\widehat u(t,\xi)$ belongs to
$W^{1,1}(0,T)$ and solves this linear ordinary differential equation;
its value at $t=0$ is $\widehat{u_0}(\xi)$ since $u\in C([0,T];L^2)$.
This gives~\eqref{eq:uhat-basic}. Since $\|e^{-tM_\varepsilon}\|\le1$,
\eqref{eq:Nhat-low-frequency}, the Cauchy--Schwarz inequality in time
and~\eqref{eq:div-uniform-bound} give~\eqref{eq:Nhat-estimate}.
\end{proof}

\subsection{Fourier splitting and the \texorpdfstring{$L^2$}{L2} decay}

Dropping the biharmonic and penalization terms
in~\eqref{eq:energy-decay-section}, we obtain
\begin{equation}
\label{eq:E0-fourier-energy}
\frac{d}{dt}E_0(t)
+
2\nu\int_{\R^3}|\xi|^2|\widehat u(t,\xi)|^2\,d\xi
\le0 .
\end{equation}
Let $B(t):=\{|\xi|\le r(t)\}$ with $r(t)^2:=\ell_0/(1+t)$, where we fix
\begin{equation}
\label{eq:a0-choice}
\ell_0:=\frac1\nu,\qquad\text{so that}\qquad a_0:=2\nu\ell_0=2 .
\end{equation}
Since
$\int|\xi|^2|\widehat u|^2\ge r^2\bigl(E_0-\int_{B(t)}|\widehat u|^2\bigr)$,
we get
\begin{equation}
\label{eq:E0-splitting}
\frac{d}{dt}E_0(t)+\frac{2}{1+t}E_0(t)\le\frac{2}{1+t}\,L(t),
\qquad
L(t):=\int_{B(t)}|\widehat u(t,\xi)|^2\,d\xi .
\end{equation}
The choice of the coefficient $2$ is a matter of convenience; any
coefficient larger than $3/2$ would do. By~\eqref{eq:Nhat-estimate}, the
identity $|B(t)|=\frac{4\pi}3r(t)^3$ and the inequality
$(a+b+c)^2\le3(a^2+b^2+c^2)$, we have
\begin{equation}
\label{eq:low-frequency-E0}
L(t)\le\frac{r(t)^3}{2\pi^2}\Bigl(\|u_0\|_1^2+r(t)^2\Phi(t)^2
+\frac14\kappa_\varepsilon^2\|u_0\|_2^2\Phi(t)\Bigr).
\end{equation}

\begin{lem}
\label{lem:iteration}
Assume that $\Phi(s)\le K(1+s)^\gamma$ for all $s\ge0$, with $K>0$ and
$\gamma\in[0,1]$. Then, for all $t\ge0$,
\begin{equation}
\label{eq:E0-bootstrap}
E_0(t)\le\Bigl[\|u_0\|_2^2+C_\nu\bigl(\|u_0\|_1^2+K^2+K\|u_0\|_2^2\bigr)\Bigr](1+t)^{-(3/2-\gamma)},
\end{equation}
where $C_\nu$ depends only on $\nu$.
\end{lem}

\begin{proof}
Multiplying~\eqref{eq:E0-splitting} by $(1+t)^2$ and integrating, we
obtain
\begin{equation}
\label{eq:E0-weighted-ineq}
(1+t)^2E_0(t)\le E_0(0)+2\int_0^t(1+s)L(s)\,ds .
\end{equation}
By~\eqref{eq:low-frequency-E0}, $r(s)^3=\nu^{-3/2}(1+s)^{-3/2}$,
$r(s)^2\Phi(s)^2\le\nu^{-1}K^2(1+s)^{2\gamma-1}\le\nu^{-1}K^2(1+s)^\gamma$
and $\kappa_\varepsilon^2\le(2\nu)^{-1}$, we get
$(1+s)L(s)\le C_\nu(\|u_0\|_1^2+K^2+K\|u_0\|_2^2)(1+s)^{\gamma-1/2}$.
Since $\int_0^t(1+s)^{\gamma-1/2}ds\le2(1+t)^{\gamma+1/2}$ and
$(1+t)^{-2}\le(1+t)^{\gamma-3/2}$, the claim follows.
\end{proof}

\begin{proof}[Proof of Theorem~\ref{thm:optimal-decay}(i)]
We apply Lemma~\ref{lem:iteration} four times. Initially
$\Phi(s)\le s\|u_0\|_2^2$, so $\gamma=1$ and $K=\|u_0\|_2^2$, which gives
$E_0(t)\le A_1(1+t)^{-1/2}$. Then $\Phi(t)\le2A_1(1+t)^{1/2}$, so
$\gamma=\frac12$ gives $E_0(t)\le A_2(1+t)^{-1}$. Then
$\Phi(t)\le A_2\log(1+t)\le4A_2(1+t)^{1/4}$, and $\gamma=\frac14$ gives
$E_0(t)\le A_3(1+t)^{-5/4}$. Then $\Phi(t)\le4A_3$, and $\gamma=0$ gives
\begin{equation}
\label{eq:E0-decay}
E_0(t)\le C_0^2(1+t)^{-3/2},\qquad C_0^2:=A_4,
\end{equation}
which is~\eqref{eq:L2-decay-final}. Each $A_j$ depends only on $\nu$,
$\|u_0\|_1$ and $\|u_0\|_2$, and it is nondecreasing in the last two
arguments. If
$A:=\|u_0\|_1+\|u_0\|_2\le1$, an induction gives $A_j\le C_\nu A^2$,
because $K^2$ and $K\|u_0\|_2^2$ are then bounded by $C_\nu A^4\le
C_\nu A^2$ at each stage. No constant depends on $\beta$ or
$\varepsilon$.
\end{proof}

\begin{rem}\rm
\label{rem:decay-uniform-epsilon}
The penalization enters the proof of
Theorem~\ref{thm:optimal-decay}(i) quantitatively only
through~\eqref{eq:div-uniform-bound}, where it has a favorable effect.
No smallness assumption on the data is needed.
\end{rem}

\subsection{Decay of the derivatives}

\begin{proof}[Proof of Theorem~\ref{thm:optimal-decay}(ii)]
Let $u_0\in L^1\cap H^2$. By~\eqref{eq:uniform-large},
$\sup_tE_1(t)+\sup_tE_2(t)\le K_1:=\Lambda_0\|u_0\|_{H^2}^2$. By
\eqref{eq:E0-decay}, $E_0(t)E_1(t)\le C_0^2K_1(1+t)^{-3/2}\le\delta_0^2$
for $t\ge t_0:=(C_0^2K_1/\delta_0^2)^{2/3}$. By the small-data part of
the proof of Theorem~\ref{thm:uniform-small-data-strong}, applied from
time $t_0$, inequality~\eqref{eq:H2-absorbed} holds on $[t_0,\infty)$,
and in particular
\begin{equation}
\label{eq:E1-diff-ineq}
E_1'+\nu D_1\le0\qquad\text{a.e. on }[t_0,\infty),
\end{equation}
\begin{equation}
\label{eq:E2-diff-ineq}
E_2'+\nu D_2\le0\qquad\text{a.e. on }[t_0,\infty).
\end{equation}
Let $r_k(t)^2:=\ell_k/(1+t)$. Since
$D_1=\int|\xi|^4|\widehat u|^2\ge r_1^2\int_{|\xi|>r_1}|\xi|^2|\widehat u|^2$,
we have
\begin{equation}
\label{eq:D1-splitting}
D_1(t)\ge r_1(t)^2E_1(t)-r_1(t)^4E_0(t),
\end{equation}
and in the same way
\begin{equation}
\label{eq:D2-splitting}
D_2(t)\ge r_2(t)^2E_2(t)-r_2(t)^4E_1(t).
\end{equation}
We fix
\begin{equation}
\label{eq:a1-choice}
\nu\ell_1=3 .
\end{equation}
Then \eqref{eq:E1-diff-ineq}, \eqref{eq:D1-splitting}
and~\eqref{eq:E0-decay} give
\begin{equation}
\label{eq:E1-splitting-ineq}
\frac{d}{dt}\Bigl[(1+t)^3E_1(t)\Bigr]\le\nu\ell_1^2C_0^2(1+t)^{-1/2}
\qquad\text{on }[t_0,\infty),
\end{equation}
and hence
$(1+t)^3E_1(t)\le(1+t_0)^3K_1+(18/\nu)C_0^2(1+t)^{1/2}$. Together with
the bound $E_1\le K_1$ on $[0,t_0]$, this gives
\begin{equation}
\label{eq:E1-decay}
\|\nabla u(t)\|_2^2=E_1(t)\le C_1^2(1+t)^{-5/2},\qquad t\ge0,
\end{equation}
with $C_1^2:=(1+t_0)^{5/2}K_1+(18/\nu)C_0^2$. Next we fix
\begin{equation}
\label{eq:a2-choice}
\nu\ell_2=4 .
\end{equation}
By~\eqref{eq:E2-diff-ineq}, \eqref{eq:D2-splitting}
and~\eqref{eq:E1-decay},
$\frac{d}{dt}\bigl[(1+t)^4E_2\bigr]\le\nu\ell_2^2C_1^2(1+t)^{-1/2}$ on
$[t_0,\infty)$, and therefore
\begin{equation}
\label{eq:E2-decay}
\|\nabla^2u(t)\|_2^2=E_2(t)\le C_2^2(1+t)^{-7/2},\qquad t\ge0,
\end{equation}
with $C_2^2:=(1+t_0)^{7/2}K_1+(32/\nu)C_1^2$. The constants depend only
on $\nu$, $\beta$, $\|u_0\|_1$ and $\|u_0\|_{H^2}$. This completes the
proof of Theorem~\ref{thm:optimal-decay}.
\end{proof}

\section{Asymptotic profile and sharpness: proof of Theorem~\ref{thm:profile}}
\label{sec:profile}

Let $u_0\in L^1\cap L^2$, let $u$ be the weak solution, and write
$\mathcal N(s):=\mathcal N(u(s))$ and $g(s):=\|\mathcal N(s)\|_1$. By
Lemma~\ref{lem:N-bounds}, \eqref{eq:E0-decay} and the energy equality,
we have
\begin{equation}
\label{eq:g-integrable}
\int_0^\infty g\,ds\le c_N\Bigl(\int_0^\infty E_0\Bigr)^{1/2}\Bigl(\int_0^\infty\|\nabla u\|_2^2\Bigr)^{1/2}
\le c_NC_0\|u_0\|_2\,\nu^{-1/2}<\infty .
\end{equation}
Also, by~\eqref{eq:E0-decay}, the Cauchy--Schwarz inequality
and~\eqref{eq:div-uniform-bound},
\begin{equation}
\label{eq:E0-integrable}
\int_0^\infty E_0(s)\,ds\le C_0^2\int_0^\infty(1+s)^{-3/2}ds=2C_0^2,
\end{equation}
\begin{equation}
\label{eq:div-term-low-frequency}
\int_0^\infty\|u(s)\|_2\|\Div u(s)\|_2\,ds\le\bigl(2C_0^2\bigr)^{1/2}\kappa_\varepsilon\|u_0\|_2,
\end{equation}
and, for every $t\ge0$,
\begin{equation}
\label{eq:tail-energy}
2\nu\int_t^\infty\|\nabla u(s)\|_2^2\,ds\le E_0(t)\le C_0^2(1+t)^{-3/2}.
\end{equation}

\begin{proof}[Proof of Theorem~\ref{thm:profile}(i)]
By~\eqref{eq:N-pointwise} and~\eqref{eq:g-integrable}, the vector
$$
n_\infty:=\int_0^\infty\!\!\int_{\R^3}\mathcal N(s)\,dx\,ds
=-\frac12\int_0^\infty\!\!\int_{\R^3} u\,\Div u\,dx\,ds
$$
is well defined, and $m_\infty=\int u_0-n_\infty$. By~\eqref{eq:div-term-low-frequency},
$|n_\infty|\le\frac12(2C_0^2)^{1/2}\kappa_\varepsilon\|u_0\|_2$.
Since $\widehat{u_0}(0)=(2\pi)^{-3/2}\int u_0$ and
$\widehat{\mathcal N(s)}(0)=(2\pi)^{-3/2}\int\mathcal N(s)$, the
Duhamel formula~\eqref{eq:uhat-basic} and~\eqref{eq:Keps} give, for
$t>0$,
$$
\widehat u(t)-\widehat{K_\varepsilon(t)m_\infty}=J_1-J_2-J_3-J_4+J_5,
$$
where, omitting the variable $\xi$,
\begin{align*}
J_1&:=e^{-tM_\varepsilon}\bigl(\widehat{u_0}-\widehat{u_0}(0)\bigr),
&
J_2&:=\int_0^{t/2}e^{-(t-s)M_\varepsilon}\bigl(\widehat{\mathcal N(s)}-\widehat{\mathcal N(s)}(0)\bigr)ds,\\
J_3&:=\int_0^{t/2}\bigl(e^{-(t-s)M_\varepsilon}-e^{-tM_\varepsilon}\bigr)\widehat{\mathcal N(s)}(0)\,ds,
&
J_4&:=\int_{t/2}^{t}e^{-(t-s)M_\varepsilon}\widehat{\mathcal N(s)}\,ds,\\
J_5&:=e^{-tM_\varepsilon}\int_{t/2}^\infty\widehat{\mathcal N(s)}(0)\,ds .
\end{align*}
We now show that $t^{3/4}\|J_i\|_2\to0$ for each $i$.

For $J_1$, the change of variables $\xi=\eta/\sqrt t$ and
$\|e^{-tM_\varepsilon}\|\le e^{-\nu t|\xi|^2}$ give
$$
t^{3/2}\|J_1\|_2^2\le\int e^{-2\nu|\eta|^2}|\widehat{u_0}(\eta/\sqrt t)-\widehat{u_0}(0)|^2d\eta,
$$
which tends to zero by dominated convergence, since $\widehat{u_0}$ is
bounded and continuous. For $J_2$, let
$$
h(\tau,s):=\tau^{3/4}\|e^{-\nu\tau|\xi|^2}(\widehat{\mathcal N(s)}-\widehat{\mathcal N(s)}(0))\|_{L^2_\xi}.
$$
As for $J_1$, we have $h(\tau,s)\le Cg(s)$, and $h(\tau,s)\to0$ as
$\tau\to\infty$ for almost every $s$. Since $t-s\ge t/2$ on $[0,t/2]$,
$t^{3/4}\|J_2\|_2\le2^{3/4}\int_0^\infty\mathbf 1_{\{s\le t/2\}}h(t-s,s)\,ds$,
which tends to zero by dominated convergence, using~\eqref{eq:g-integrable}.

For $J_3$, the eigenvalues $\lambda$ of $M_\varepsilon(\xi)$ satisfy
$\lambda\ge\nu|\xi|^2$, and for $0\le s\le t/2$,
$0\le e^{-(t-s)\lambda}-e^{-t\lambda}\le s\lambda e^{-t\lambda/2}
\le\frac4e\,\frac st\,e^{-t\lambda/4}$. Hence
$\|e^{-(t-s)M_\varepsilon}-e^{-tM_\varepsilon}\|\le\frac4e\frac st\,e^{-\nu t|\xi|^2/4}$,
and
$t^{3/4}\|J_3\|_2\le C\nu^{-3/4}\int_0^{t/2}\frac st\,g(s)\,ds\to0$ by
dominated convergence. For $J_5$,
$t^{3/4}\|J_5\|_2\le C\nu^{-3/4}\int_{t/2}^\infty g\,ds\to0$.

For $J_4$ we use the biharmonic part of the symbol, namely
$\|e^{-\tau M_\varepsilon(\xi)}\|\le e^{-\beta\tau|\xi|^4}$ and
$\|e^{-\beta\tau|\xi|^4}\|_{L^2_\xi}=c\,(\beta\tau)^{-3/8}$. Hence, by the
Cauchy--Schwarz inequality in $s$,
$$
\|J_4\|_2\le C\beta^{-3/8}\int_{t/2}^t(t-s)^{-3/8}g(s)\,ds
\le C\beta^{-3/8}t^{1/8}\Bigl(\int_{t/2}^\infty g^2\,ds\Bigr)^{1/2}.
$$
For $s\ge t/2$, $g(s)^2\le c_N^2E_0(s)\|\nabla u(s)\|_2^2\le
c_N^2C_0^2(1+t/2)^{-3/2}\|\nabla u(s)\|_2^2$, so by~\eqref{eq:tail-energy},
$\|J_4\|_2\le C\,t^{1/8}(1+t)^{-3/2}$ and $t^{3/4}\|J_4\|_2\to0$.

This proves the first limit in~\eqref{eq:profile}. The second limit
follows from the first one and Lemma~\ref{lem:kernel}.
\end{proof}

\begin{proof}[Proof of Theorem~\ref{thm:profile}(ii)]
Set $m_0:=\int u_0$ and $\widehat m:=\widehat{u_0}(0)=(2\pi)^{-3/2}m_0$,
and let
$$
\omega(\rho):=\sup_{|\xi|\le\rho}|\widehat{u_0}(\xi)-\widehat m|,
$$
which tends to zero as $\rho\to0$. By~\eqref{eq:semigroup-split},
$\widehat\Pp(\xi)e^{-tM_\varepsilon(\xi)}=e^{-ta(\xi)}\widehat\Pp(\xi)$,
so the Duhamel formula gives
$$
\widehat\Pp\,\widehat u(t)=e^{-ta}\widehat\Pp\,\widehat{u_0}
-\int_0^te^{-(t-s)a}\widehat\Pp\,\widehat{\mathcal N(s)}\,ds ,
$$
in which $\varepsilon$ appears only through $u$. Since
$|\widehat\Pp|\le1$, $e^{-ta}\le1$, \eqref{eq:Nhat-low-frequency},
\eqref{eq:E0-integrable} and~\eqref{eq:div-term-low-frequency},
$$
\bigl|\widehat\Pp(\xi)\widehat u(t,\xi)-e^{-ta(\xi)}\widehat\Pp(\xi)\widehat m\bigr|
\le\omega(|\xi|)+(2\pi)^{-3/2}\Bigl(2C_0^2|\xi|+2^{-1/2}C_0\kappa_\varepsilon\|u_0\|_2\Bigr).
$$
Let $\theta:=e^{-2}/(3\sqrt6)$ and $\Lambda:=3\sqrt3\,e^2$, so that
\eqref{eq:nondeg} means $(2\pi)^{-3/2}2^{-1/2}C_0\kappa_\varepsilon\|u_0\|_2\le\theta|\widehat m|$.
Let $T_1\ge\max\{1,\beta/\nu^2\}$ be such that, for $t\ge T_1$ and
$\rho_t:=(\nu t)^{-1/2}$, we have $\omega(\rho_t)\le\theta|\widehat m|$
and $2C_0^2\rho_t\le\theta|m_0|$. For $t\ge T_1$ and $|\xi|\le\rho_t$ we
have $ta(\xi)\le1+\beta/(\nu^2t)\le2$, hence
$|\widehat\Pp\widehat u(t,\xi)|\ge e^{-2}|\widehat\Pp(\xi)\widehat m|-3\theta|\widehat m|$
and, by $(x-y)_+^2\ge\frac12x^2-y^2$,
$$
|\widehat u(t,\xi)|^2\ge|\widehat\Pp(\xi)\widehat u(t,\xi)|^2
\ge\frac12e^{-4}|\widehat\Pp(\xi)\widehat m|^2-9\theta^2|\widehat m|^2 .
$$
Integrating over $\{|\xi|\le\rho_t\}$ and using that the angular average
of $|\widehat\Pp(\omega)\widehat m|^2$ is $\frac23|\widehat m|^2$, we get
$$
\|u(t)\|_2^2\ge\frac{4\pi}{3}\rho_t^3|\widehat m|^2\Bigl(\frac13e^{-4}-9\theta^2\Bigr)
=\frac{4\pi}{3}\cdot\frac{e^{-4}}{6}\,(2\pi)^{-3}|m_0|^2(\nu t)^{-3/2},
$$
which is the claim with $c_*^2:=e^{-4}/(36\pi^2)$. Note that $T_1$
depends only on $\nu$, $\beta$, $C_0$, $m_0$ and $\omega$.
\end{proof}

\begin{proof}[Proof of Theorem~\ref{thm:profile}(iii)]
By~\eqref{eq:Nhat-estimate} and $\Phi\le2C_0^2$, for $|\xi|\le1$ and all
$t$ we have $|\widehat u(t,\xi)|\le U:=(2\pi)^{-3/2}(\|u_0\|_1+2C_0^2+C_0\kappa_\varepsilon\|u_0\|_2)$,
and $U$ is bounded independently of $\varepsilon$. For $r\in(0,1]$,
$E_0(t)\le\frac{4\pi}3r^3U^2+r^{-2}E_1(t)$. For $t\ge T_1$, part (ii)
gives $E_0(t)\ge c_*^2|m_0|^2(\nu t)^{-3/2}$; choosing
$r=r_t:=\rho(\nu t)^{-1/2}$ with $\frac{4\pi}3\rho^3U^2=\frac12c_*^2|m_0|^2$,
which satisfies $r_t\le1$ for $t\ge T_2:=\max\{T_1,\rho^2/\nu\}$, we
obtain $E_1(t)\ge\frac12c_*^2|m_0|^2\rho^2(\nu t)^{-5/2}$. Finally
$E_1^2\le E_0E_2$ by~\eqref{eq:ineq-interp}, so
$E_2(t)\ge E_1(t)^2/E_0(t)\ge c\,t^{-5}(1+t)^{3/2}/C_0^2$, which gives
the lower bound for $k=2$. This completes the proof of
Theorem~\ref{thm:profile}.
\end{proof}

\section{The incompressible limit: proof of Theorem~\ref{thm:incompressible-limit}}
\label{sec:limit}

We first recall the global theory for~\eqref{eq:HNS}. Global regularity
for arbitrary data is classical~\cite{Lions1969}; we need the bounds in
the following form.

\begin{prop}
\label{prop:HNS}
Let $v_0\in H^2$ with $\Div v_0=0$. Problem~\eqref{eq:HNS} with
$v(0)=v_0$ has a unique global strong solution
$v\in C([0,\infty);H^2)\cap L^2_{\mathrm{loc}}([0,\infty);H^4)$, and $v$
satisfies~\eqref{eq:large-data-bound} with $u$, $u_0$ replaced by $v$,
$v_0$ (and without the penalization term).
\end{prop}

\begin{proof}
Applying $\Pp$, problem~\eqref{eq:HNS} becomes
$\partial_tv-\nu\Delta v+\beta\Delta^2v+\Pp((v\cdot\nabla)v)=0$, whose
linear part generates $e^{t(\nu\Delta-\beta\Delta^2)}$, which preserves
divergence-free fields. The proof of
Theorem~\ref{thm:uniform-small-data-strong} applies verbatim: $\Pp$ is
self-adjoint, commutes with derivatives and is a contraction on every
$H^s$, so that the energy identities~\eqref{eq:H2-diff-alpha} hold with
$\mathcal N(v)=(v\cdot\nabla)v$, and Lemmas~\ref{lem:cancel}
and~\ref{lem:Hk-key} apply because $\Div v=0$.
\end{proof}

Let now $u=u^\varepsilon$ be the strong solution of~\eqref{Main-eq}, $v$
the solution of~\eqref{eq:HNS} with $v(0)=\Pp u_0$, and write
$u=\Pp u+\Qq u$, $q:=\Qq u$, $w:=\Pp u-v$. Since
$\|\Pp u_0\|_{H^k}\le\|u_0\|_{H^k}$ and $\Lambda_0$ is increasing in
$\|u_0\|_2$, Theorem~\ref{thm:uniform-small-data-strong} and
Proposition~\ref{prop:HNS} give a constant $K$, depending only on
$\nu$, $\beta$ and $\|u_0\|_{H^2}$, with
\begin{equation}
\label{eq:limit-K}
\sup_{t\ge0}\bigl(\|u(t)\|_{H^2}^2+\|v(t)\|_{H^2}^2\bigr)
+\int_0^\infty\bigl(\|\nabla u\|_{H^1}^2+\|\nabla v\|_{H^1}^2\bigr)dt\le K .
\end{equation}
Gradient fields satisfy $|\xi||\widehat q|=|\xi\cdot\widehat u|$, hence
\begin{equation}
\label{eq:q-identities}
\begin{gathered}
\|\nabla q\|_2=\|\Div u\|_2,\qquad \Div q=\Div u,\qquad
\|q\|_6\le C\|\Div u\|_2,\\
\int_0^\infty\|\Div u\|_2^2\,dt\le\frac\varepsilon2\|u_0\|_2^2 .
\end{gathered}
\end{equation}

\begin{lem}
\label{lem:w-estimate}
$\sup_{t\ge0}\|w(t)\|_2^2+\nu\int_0^\infty\|\nabla w\|_2^2\,dt\le C\varepsilon\|u_0\|_2^2$,
with $C$ depending only on $\nu$, $\beta$ and $\|u_0\|_{H^2}$.
\end{lem}

\begin{proof}
Applying $\Pp$ to~\eqref{Main-eq} removes the penalization, since
$\Pp\nabla\Div=0$. Subtracting the projected form of~\eqref{eq:HNS},
$$
\partial_tw-\nu\Delta w+\beta\Delta^2w+\Pp\bigl[\mathcal N(u)-(v\cdot\nabla)v\bigr]=0,
\qquad w(0)=0 .
$$
Since $u=v+w+q$, we have
$\mathcal N(u)-(v\cdot\nabla)v=(w\cdot\nabla)u+(q\cdot\nabla)u+(v\cdot\nabla)w+(v\cdot\nabla)q+\frac12u\,\Div u$.
We test this equation with $w=\Pp w$, which is justified as in
Section~\ref{sec:strong-solutions}. The term with $(v\cdot\nabla)w$
vanishes because $\Div v=0$. Next,
$|\int(w\cdot\nabla)u\cdot w|\le\|w\|_3\|\nabla u\|_2\|w\|_6
\le C\|w\|_2^{1/2}\|\nabla w\|_2^{3/2}\|\nabla u\|_2
\le\frac\nu4\|\nabla w\|_2^2+C\nu^{-3}\|\nabla u\|_2^4\|w\|_2^2$.
Integrating by parts,
$\int(q\cdot\nabla)u\cdot w=-\int(\Div q)(u\cdot w)-\int(q\cdot\nabla)w\cdot u$
and $\int(v\cdot\nabla)q\cdot w=-\int(v\cdot\nabla)w\cdot q$. By
H\"older's inequality, \eqref{eq:ineq-sobolev}
and~\eqref{eq:q-identities}, these two terms and
$\frac12\int(\Div u)(u\cdot w)$ are bounded by
$C(\|u\|_{H^1}+\|v\|_{H^1})\|\Div u\|_2\|\nabla w\|_2$, using
$\|u\|_3\le C\|u\|_{H^1}$, $\|w\|_6\le C\|\nabla w\|_2$ and
$\|q\|_6\le C\|\Div u\|_2$. With~\eqref{eq:limit-K} we obtain
$$
\frac{d}{dt}\|w\|_2^2+\nu\|\nabla w\|_2^2\le C\nu^{-3}\|\nabla u\|_2^4\|w\|_2^2+C\nu^{-1}K\|\Div u\|_2^2 .
$$
Since $\int_0^\infty\|\nabla u\|_2^4\le\sup_t\|\nabla u\|_2^2\int_0^\infty\|\nabla u\|_2^2\le K^2$,
Gronwall's inequality and~\eqref{eq:q-identities} give the claim.
\end{proof}

\begin{lem}
\label{lem:q-estimate}
$\sup_{t\ge0}\|q(t)-\ell^\varepsilon(t)\|_2\le C\sqrt\varepsilon\|u_0\|_2$,
with $C$ depending only on $\nu$, $\beta$ and $\|u_0\|_{H^2}$.
\end{lem}

\begin{proof}
The mild formulation~\eqref{eq:mild-strong} holds for every $t\ge0$:
both sides are strong solutions of the same linear inhomogeneous problem
with forcing $-\mathcal N(u)\in L^2(0,T;L^2)$ on every $[0,T]$, and such
solutions are unique. By~\eqref{eq:mild-strong} and since $\Qq$ commutes with
$S_\varepsilon$, $q(t)-\ell^\varepsilon(t)=-\int_0^t\Qq S_\varepsilon(t-s)\mathcal N(s)\,ds$.
By~\eqref{eq:semigroup-split},
$\widehat\Qq e^{-\tau M_\varepsilon}=e^{-\tau(a+|\xi|^2/\varepsilon)}\widehat\Qq$,
whose norm is at most $e^{-\tau|\xi|^2/\varepsilon}$. We use
\eqref{eq:N-decomp}. For the divergence part, with $F:=u\otimes u$, the
Cauchy--Schwarz inequality in $s$ gives, for each $\xi$,
\begin{align*}
\int_0^te^{-(t-s)|\xi|^2/\varepsilon}|\xi||\widehat F(s,\xi)|\,ds
&\le\Bigl(\int_0^te^{-2(t-s)|\xi|^2/\varepsilon}|\xi|^2ds\Bigr)^{1/2}\Bigl(\int_0^\infty|\widehat F(s,\xi)|^2ds\Bigr)^{1/2}\\
&\le\Bigl(\frac\varepsilon2\Bigr)^{1/2}\Bigl(\int_0^\infty|\widehat F(s,\xi)|^2ds\Bigr)^{1/2}.
\end{align*}
Integrating in $\xi$, the $L^2$ norm of this contribution is at most
$(\varepsilon/2)^{1/2}(\int_0^\infty\|u\|_\infty^2\|u\|_2^2\,ds)^{1/2}$.
For the compressible part, Minkowski's inequality and~\eqref{eq:q-identities}
give the bound
$$
\frac12\int_0^\infty\|u\|_\infty\|\Div u\|_2\,ds
\le\frac12\Bigl(\int_0^\infty\|u\|_\infty^2\,ds\Bigr)^{1/2}\Bigl(\frac\varepsilon2\Bigr)^{1/2}\|u_0\|_2 .
$$
Finally $\|u\|_2\le\|u_0\|_2$ and, by~\eqref{eq:ineq-agmon}
and~\eqref{eq:limit-K},
$\int_0^\infty\|u\|_\infty^2\le C\int_0^\infty\|\nabla u\|_2\|\nabla^2u\|_2\le CK$.
\end{proof}

\begin{proof}[Proof of Theorem~\ref{thm:incompressible-limit}]
Since $u-v-\ell^\varepsilon=w+(q-\ell^\varepsilon)$, part (i) follows
from Lemmas~\ref{lem:w-estimate} and~\ref{lem:q-estimate}, together with
$\int_0^\infty\|\nabla(u-v)\|_2^2\le2\int_0^\infty\|\nabla w\|_2^2+2\int_0^\infty\|\Div u\|_2^2$
and~\eqref{eq:q-identities}. For (ii), we note that~\eqref{eq:semigroup-split} gives
$|\widehat{\ell^\varepsilon}(t,\xi)|\le e^{-t|\xi|^2/\varepsilon}|\widehat\Qq\widehat{u_0}(\xi)|$,
hence $\|\ell^\varepsilon(t)\|_2\le\|e^{(t/\varepsilon)\Delta}\Qq u_0\|_2$,
and, if $u_0\in L^1$, $|\widehat\Qq\widehat{u_0}|\le(2\pi)^{-3/2}\|u_0\|_1$
gives the last bound as in~\eqref{eq:semigroup-L2}.

For (iii), we apply $\Qq$ to~\eqref{Main-eq}. Since $\Qq\nabla\Div=\nabla\Div$,
we obtain
$$
\nabla\Bigl(-\frac1\varepsilon\Div u\Bigr)=-\partial_tq+\nu\Delta q-\beta\Delta^2q-\Qq\mathcal N(u).
$$
By~\eqref{eq:q-identities}, $\|q\|_{L^2(0,\infty;L^6)}\le C\sqrt\varepsilon\|u_0\|_2$,
so the first three terms tend to zero in $\mathcal D'$. Next,
$\mathcal N(u)-(v\cdot\nabla)v=((u-v)\cdot\nabla)u+(v\cdot\nabla)(u-v)+\frac12u\,\Div u$.
For $T>0$, the first term tends to zero in $L^1(0,T;L^{3/2})$, since
$\|((u-v)\cdot\nabla)u\|_{3/2}\le C\|u-v\|_2\|\nabla^2u\|_2$ and
$\int_0^T\|u-v\|_2^2\le2\int_0^T\|\ell^\varepsilon\|_2^2+CT\varepsilon\to0$
by (i), (ii) and dominated convergence; the last two terms tend to zero
in $L^2(0,\infty;L^{6/5})$, since they are bounded by
$C\|v\|_3\|\nabla(u-v)\|_2$ and $C\|u\|_3\|\Div u\|_2$. As $\Qq$ is
bounded on $L^{3/2}$ and on $L^{6/5}$, we get
$\Qq\mathcal N(u^\varepsilon)\to\Qq((v\cdot\nabla)v)$ in
$\mathcal D'$. Finally, applying $\Qq$ to~\eqref{eq:HNS} gives
$\nabla\pi=-\Qq((v\cdot\nabla)v)$, which proves (iii).

It remains to prove the stated consequences. The function
$t\mapsto\|e^{(t/\varepsilon)\Delta}\Qq u_0\|_2$ is nonincreasing, and
for each $t>0$ it tends to zero as $\varepsilon\to0$ by dominated
convergence. This gives the convergence in $L^\infty(t_0,\infty;L^2)$
and, by dominated convergence in $t$, in $L^p(0,T;L^2)$. If
$\Div u_0=0$, then $\Qq u_0=0$ and $\ell^\varepsilon\equiv0$. This
completes the proof of Theorem~\ref{thm:incompressible-limit}.
\end{proof}

\section{Conclusion and open questions}
\label{sec:conclusion}

In this paper, we studied a penalized Navier--Stokes system with
biharmonic viscosity in $\R^3$ and obtained estimates that are uniform in
the penalization parameter $\varepsilon$. We proved global well-posedness
in $L^2$, global regularity for $H^2$ data of arbitrary size, decay
estimates at the heat rates together with the leading asymptotic profile,
and global-in-time convergence to the incompressible hyperviscous
Navier--Stokes equations as $\varepsilon\to0$. We conclude with some
questions that are not addressed in this paper. None of the statements
in this section is claimed as a result.

\emph{The limit $\beta\to0$.} The constant $C_0$
in~\eqref{eq:L2-decay-final} does not depend on $\beta$. For $\beta=0$,
system~\eqref{Main-eq} is the endpoint $a=1$ of the family
of~\cite{Nonl2003}, and we do not know whether its solutions remain
regular in three dimensions.
It would be interesting to determine which of our uniform bounds survive
in the limit $\beta\to0$.

\emph{Interchanging the limits $t\to\infty$ and $\varepsilon\to0$.}
Theorem~\ref{thm:profile} gives
$$
\lim_{\varepsilon\to0}\lim_{t\to\infty}t^{3/2}\|u^\varepsilon(t)\|_2^2
=\frac23(8\pi\nu)^{-3/2}\Bigl|\int u_0\Bigr|^2 .
$$
At the linear level, for $u_0\in L^1\cap L^2$, the argument of
Lemma~\ref{lem:kernel} (together with the continuity of $\widehat{u_0}$
at the origin) shows that
$t^{3/2}\|e^{t(\nu\Delta-\beta\Delta^2)}\Pp u_0\|_2^2$ has the same
limit as $t\to\infty$.
We expect, but have not proved, that the solution $v$
of~\eqref{eq:HNS} has the same asymptotic energy; note that $\Pp
u_0\notin L^1$ in general, so this does not follow from the known
theory.

\emph{Higher-order asymptotics and other norms.} Natural continuations
include asymptotic expansions of the type obtained by Fujigaki and
Miyakawa~\cite{FM2001}, decay in $L^p$, and a description of the initial
layer in stronger norms.

\emph{Bounded domains.} In the presence of boundaries, the penalty
method creates boundary layers, and the analysis of~\eqref{Main-eq}
with physical boundary conditions remains an open problem.

\section*{Declarations}
\noindent\textbf{Funding.} This research received no external funding.

\noindent\textbf{Conflict of interest.} The authors declare that they have no conflicts of interest.

\noindent\textbf{Data availability.} No datasets were generated or analyzed during the current study.


\end{document}